\documentclass[a4paper,12pt]{amsart}
\usepackage{amssymb,amsfonts,latexsym,mathtools,graphicx,comment,hyperref}
\usepackage[utf8]{inputenc} 
\usepackage[T1]{fontenc}    
\usepackage{lmodern}    
\usepackage{xcolor}        
\usepackage{amsthm,enumitem,appendix}
\usepackage{comment}
\usepackage{mathrsfs}
\usepackage{todonotes}

\newtheorem{theorem}{Theorem}[section]
\newtheorem{proposition}{Proposition}[section]
\newtheorem{lemma}{Lemma}[section]

\newtheorem{definition}{Definition}[section]
\newtheorem{remark}{Remark}[section]
\DeclareMathOperator{\supp}{supp}
\renewcommand{\[}{\left[}

\newcommand{\eps}{\varepsilon}
\newcommand{\dvg}{\mathrm{~d}V_{\mathbb B^N}}

\newcommand{\grad}{\nabla}
\newcommand{\R}{\mathbb{R}}

\newcommand{\bn}{\mathbb{B}^N}

\newcommand{\abs}[1]{\left\vert#1\right\vert}

\newcommand{\la}{\lambda}
\newcommand{\De}{\Delta}

\newcommand{\BN}{\mathbb{B}^N}
\newcommand{\RN}{\mathbb{R}^N}

\newcommand{\Sph}{\mathbb{S}^{N-1}}
\newcommand{\dV}{\,dV_{\BN}}
\newcommand{\Ks}{\mathcal{K}_s}
\newcommand{\lap}{\Delta_{\BN}}
\newcommand{\gradB}{\nabla_{\BN}}
\newcommand{\dist}{d}
\newcommand{\cstar}{2^{\ast}}
\newcommand{\Sloc}{S_\lambda^{\mathrm{loc}}}
\newcommand{\Smix}{S_{\lambda,s}}
\newcommand{\Hh}{H^{1}(\BN)}
 \catcode`\@=11
 
 \makeatletter \@addtoreset{equation}{section} \makeatother

\begin{document}

\title[Brezis--Nirenberg Problem for Mixed Operators on $\mathbb{B}^N$]{A Brezis--Nirenberg Type Problem for a Mixed Local-Nonlocal Operator on the Hyperbolic Space}

\author[D.~Gupta]{Diksha Gupta$^{*}$}
\author[K.~Sreenadh]{K.~Sreenadh}
\thanks{Department of Mathematics, Indian Institute of Technology Delhi, Hauz Khas, New Delhi 110016, India.\\
\textit{Emails:} \texttt{dikshagupta1232@gmail.com} (D.~Gupta), \texttt{sreenadh@maths.iitd.ac.in} (K.~Sreenadh).\\
$^{*}$\,Corresponding author.}

\begin{abstract}
This paper studies nonlinear elliptic equations involving the operator
$-\Delta_{\mathbb{B}^N} + (-\Delta_{\mathbb{B}^N})^s-\lambda$ on the hyperbolic space
$\mathbb{B}^N$, where $0 < s < 1$ and $\lambda<\frac{(N-1)^2}{4}$. We prove that the best
Sobolev constant of this mixed local-nonlocal operator equals the Euclidean Sobolev
constant $S$ and is never attained, whereas the corresponding purely local constant may lie
strictly below $S$. As a consequence, the purely critical problem has no ground state and no
nontrivial solution of small $L^{2^*}$ norm. For the subcritical problem, and for the
critical problem with a subcritical perturbation $\mu|u|^{p-1}u$, we obtain nonnegative
radial solutions by variational methods, depending on $p$ and $\mu$. Along the way, we
prove the continuous embedding $H^1(\mathbb{B}^N)\hookrightarrow H^s(\mathbb{B}^N)$ and the
equivalence of the spectral and the pointwise definitions of $(-\Delta_{\mathbb{B}^N})^s$.

\end{abstract}
\maketitle
%\tableofcontents

\section{Introduction}

In this work, we study a Brezis--Nirenberg type problem for the mixed local-nonlocal
operator $-\Delta_{\bn}+(-\Delta_{\bn})^s$ on the hyperbolic space $\bn$. We begin with
the subcritical problem
\begin{equation}
\begin{cases}
-\Delta_{\bn} u + (-\Delta_{\bn})^s u - \lambda u = |u|^{p-1}u, & u \in H^1(\bn),
\end{cases}
\tag{$P_{\lambda,s}$}\label{mainEq}
\end{equation}
where $ \Delta_{\bn} $ is the Laplace–Beltrami operator and $ (-\Delta_{\bn})^s $ is the fractional Laplacian of order $ 0 < s < 1 $, both defined on the hyperbolic space $ \bn $. The parameter $ \lambda < \frac{(N-1)^2}{4} $, the bottom of the $ L^2 $-spectrum of $ -\Delta_{\bn} $, and the nonlinearity exponent satisfies $ 1 < p < 2^*-1 = \frac{N+2}{N-2} $ for $ N \geq 3 $. The natural variational setting for this problem is the Sobolev space $ H^1(\bn) $ (refer to Section~\ref{PrelimSection}).

\smallskip
In addition to this subcritical setting, we are interested in the critical regime.
We first consider the purely critical problem
\begin{equation}
\begin{cases}
-\Delta_{\bn} u + (-\Delta_{\bn})^s u - \lambda u = |u|^{2^*-2}u, & u \in H^1(\bn),
\end{cases}
\tag{$P^*_{\lambda,s}$}\label{mainEqCrit}
\end{equation}
and then its perturbation
\begin{equation}
\begin{cases}
-\Delta_{\bn} u + (-\Delta_{\bn})^s u - \lambda u = |u|^{2^*-2}u + \mu|u|^{p-1}u, & u \in H^1(\bn),
\end{cases}
\tag{$\tilde{P}_{\lambda,s,\mu}$}\label{mainEq2}
\end{equation}
where $\mu\in\R$ and the term $\mu|u|^{p-1}u$ acts as a lower-order perturbation of the
critical power nonlinearity $ |u|^{2^*-2}u $. Our interest lies in analyzing the existence
of solutions in the presence of the geometric and analytic complexities introduced by the
nonlocal structure in the hyperbolic space.

\smallskip
Semilinear elliptic equations of the type involving local operators such as the Laplacian have been extensively studied in the Euclidean setting. In their seminal work \cite{BrezisNirenberg1983}, Brezis and Nirenberg considered
$-\Delta u=u^{2^*-1}+\lambda u$, $u>0$, in a bounded domain $\Omega\subset\RN$ with
$u=0$ on $\partial\Omega$. They showed that, although the embedding
$H_0^1(\Omega)\hookrightarrow L^{2^*}(\Omega)$ is not compact, the lower-order term
$\lambda u$ can restore existence. The key step is to show that the
associated Sobolev quotient lies strictly below the best Sobolev constant $S$, which
restores compactness of Palais--Smale sequences below the level $\frac1NS^{N/2}$. Foundational results concerning the existence, regularity, and qualitative properties of solutions to such problems have been developed further in various contexts (see, e.g., \cite{BrezisNirenberg1983, Struwe1984, Lions1985}). Parallel to these developments, there has been a growing interest in nonlocal problems governed by the fractional Laplacian, which naturally arises in models related to reaction-diffusion problems, quantum mechanics, image processing, and mathematical finance. Such equations are governed by integro-differential operators that reflect long-range interactions and lack the strong locality of classical Laplacian-type operators, leading to new analytical challenges and rich mathematical phenomena (see \cite{DiNezzaPalatucciValdinoci2012, ServadeiValdinoci2014, RosOton2014,ChenFelmerQuaas2019, FallValdinoci2022,servadei2015brezis,ByeonScalarField}). 

\smallskip
More recently, a growing body of research has focused on elliptic equations that combine local and nonlocal effects-so-called mixed-order or mixed-local-nonlocal problems, in relation to the possible existence of positive solutions for critical problems and in connection with possible optimizers of suitable Sobolev inequalities. For example, in \cite{biagi}, the authors investigate a Brezis--Nirenberg type problem
for a mixed local-nonlocal operator of the form
\begin{equation*}
-\Delta u + (-\Delta)^s u = u^{2^*-1}+\lambda u^{p},\quad u\ge0 \quad \text{in } \Omega,\qquad
u = 0 \quad \text{in } \mathbb{R}^N \setminus \Omega,
\end{equation*}
where $\Omega \subset \mathbb{R}^N$ is a bounded open set, $\lambda \in \mathbb{R}$,
$p\in[1,2^*-1)$, and $2^* = \frac{2N}{N-2}$ is the critical Sobolev exponent. They study the
associated Sobolev inequality and establish existence and nonexistence results for this
problem (see also \cite{VariationalForMixed, BonheureVanSchaftingen2020,
FractionalValdinoci1, ValdinociMaximum, Garain1}).

\smallskip

There has also been a substantial shift toward studying elliptic equations on non-Euclidean spaces, particularly due to their applicability in modeling various physical and geometric phenomena. A central motivation in this direction is to understand how the curvature of the underlying space, especially the negative curvature of hyperbolic geometry, affects the existence, uniqueness, and qualitative behavior of solutions. Seminal contributions in this area include the work of Sandeep and Mancini \cite{MS, ManciniHardySobolev}, who investigated semilinear equations involving the Laplace–Beltrami operator on the hyperbolic space.   
In \cite{MS}, Sandeep and Mancini proved the existence and uniqueness (up to isometries) of finite-energy positive solutions to the homogeneous elliptic equation
\begin{equation}\label{hsm}
	-\Delta_{\mathbb{B}^{N}} u \, - \, \lambda u \, = \, u^{p}, \quad u \in H^{1}\left(\mathbb{B}^{N}\right),
\end{equation}
where $\lambda \leq \frac{(N-1)^2}{4},$ $1 < p \leq \frac{N+2}{N-2}$ if $N \geq 3; $ $1 < p < \infty$ if $N=2.$ They established that, in the subcritical case and for $p > 1$, equation \eqref{hsm} admits a positive solution if and only if $\lambda < \frac{(N-1)^2}{4}$, and that such solutions are unique modulo hyperbolic isometries except possibly when $N = 2$ and $\lambda > \frac{2(p+1)}{(p+3)^2}$. 

\smallskip
Following their work, subsequent studies have deepened the analysis of \eqref{hsm} in various directions. For instance, \cite{MousomiUTF, DG1, DG2} addressed the existence of sign-changing solutions, compactness properties, and non-degeneracy, while \cite{BFG, BGGV} explored infinite-energy solutions and their asymptotic profiles. Building upon this framework, the authors of this article, in collaboration with Ganguly, Bhakta, and Sahoo, have analyzed non-radial perturbations of \eqref{hsm} and contributed to the understanding of Palais-Smale decomposition, multiplicity, and asymptotic behavior of solutions (see \cite{GGS2, asympOur, OurCCM, OurPRSE, bhakta2024poincaresobolevequationscriticalexponent}).
Moreover, Sandeep and Mancini showed that equation \eqref{hsm} arises as the Euler–Lagrange equation of the Hardy–Sobolev–Maz'ya-type inequality in the hyperbolic space, which has led to extensive work on related inequalities and higher-order conformally covariant operators, including Paneitz-type operators studied in the hyperbolic setting (see \cite{Lu1, Lu2, PanietzHyperbolic}).

\smallskip

A natural extension of the above discussion involves extending the analysis to nonlocal problems in the hyperbolic space. The fractional Laplacian $\left(-\Delta_{\mathbb{B}^N}\right)^s$, for $0 < s < 1$, on the hyperbolic space $\mathbb{B}^N$ can be defined using the standard functional calculus. However, in \cite{BanicaEtalFractionalLaplacianHyperbolic}, the authors showed that $\left(-\Delta_{\mathbb{B}^N}\right)^s$ can be interpreted as the Dirichlet-to-Neumann operator associated with a degenerate elliptic extension problem, analogous to the Caffarelli–Silvestre extension for the Euclidean setting \cite{CafferelliSilvsitreEuclidean}, and its generalization to manifolds \cite{Caffsilmanifold}. In addition, they provided an explicit integral representation of the fractional Laplacian in the hyperbolic space as a principal value (P.V.) singular integral, expressed as a convolution with a well-behaved kernel $\mathcal{K}_s$ (up to a normalizing constant):
\begin{equation*}
\left(-\Delta_{\mathbb{B}^N}\right)^s u(x) = \text{P.V.} \int_{\mathbb{B}^N} (u(x) - u(\xi))\;\mathcal{K}_s(d(x, \xi))\;\mathrm{d}\xi,
\end{equation*}
where the kernel $\mathcal{K}_s(d(x, \xi))$ is comparable to $\frac{1}{d(x,\xi)^{N+2s}}$ as $d(x,\xi) \to 0$, and exhibits exponential decay as $d(x, \xi) \to \infty$ (see \cite[Theorems 2.4 and 2.5]{BanicaEtalFractionalLaplacianHyperbolic}; also see Section~\ref{PrelimSection}). Notably, as $d(x, \xi) \to 0$, the structure of this kernel is reminiscent of the Euclidean singular kernel $\frac{1}{|x - y|^{N+2s}}$ that appears in the classical definition of the fractional Laplacian.

\smallskip

Further refinements appear in \cite{harnackinequalityfractionallaplaciantype}, where the authors, focusing on the case $N=3$, determined the explicit value of the normalizing constant, which proves essential in certain applications. This representation was later extended to the fractional $p$-Laplacian in the hyperbolic space in \cite[Theorem 1.2]{fractionalplaplacianhyperbolicspaces}, where the authors incorporated the normalizing constant into the generalized formulation. A broader perspective on fractional Laplacians on manifolds is given in \cite{fracLaplClosedManifolds}, where several equivalent definitions are established and shown to coincide (up to multiplicative constants) on closed Riemannian manifolds (i.e., compact without boundary). These approaches often rely on the Fourier transform on hyperbolic space, which serves as a key analytical tool in \cite{BanicaEtalFractionalLaplacianHyperbolic, harnackinequalityfractionallaplaciantype}. However, in nonlinear settings \cite{fractionalplaplacianhyperbolicspaces} where the Fourier transform is unavailable, alternative techniques such as the heat kernel associated with the Laplace operator on $\mathbb{B}^N$ are employed.

\smallskip
These foundational results provide a rigorous framework for analyzing semilinear problems involving the fractional Laplacian on hyperbolic space. In particular, when the nonlinearity is derived from a double-well potential, one can establish the existence and uniqueness of layer solutions, along with certain symmetry properties, as shown in \cite{LayerSolutionMaria}. Additionally, in \cite{BinlinZhangManifold}, the authors addressed fractional $p$-Laplacian equations with homogeneous Dirichlet boundary conditions on compact Riemannian manifolds. 

\smallskip
Motivated by the above developments, we turn our attention to nonlinear elliptic problems
involving mixed-order local and nonlocal operators in the hyperbolic space. Our broader
objective is to study the Brezis--Nirenberg problem involving a mixed local-nonlocal
operator in the hyperbolic space. However, we observe a notable gap in the literature.
While significant progress has been made for such problems in the Euclidean setting
\cite{biagi}, the study of elliptic equations involving the fractional Laplacian in
hyperbolic space remains largely unexplored. Moreover, two features of the hyperbolic
setting change the picture. First, the term $-\lambda u$ does not lower the Sobolev quotient
below $S$ as it does in \cite{BrezisNirenberg1983}. With the nonlocal term present, the
best constant equals $S$ for every $\lambda<\frac{(N-1)^2}{4}$
(Theorem~\ref{thm:sobolevConstants}). Second, the Pohozaev-type argument, which for the local operator on $\bn$ works through the
conformal change to the Euclidean unit ball, is not available, because the nonlocal term
does not transform well under this change. To bridge this gap, we initiate the investigation
of such problems by considering the subcritical problem \eqref{mainEq}, the purely critical
problem \eqref{mainEqCrit}, and its perturbation \eqref{mainEq2}, in which the gain below the
level $\frac1NS^{N/2}$ comes from the subcritical term $\mu|u|^{p-1}u$. In our recent work
\cite{gupta2024existencesymmetryregularityground}, we have also examined the existence of positive
solutions to a nonlinear Choquard equation involving the Green kernel of the fractional
operator $(-\Delta_{\mathbb{B}^N})^{-\alpha/2}$ in hyperbolic space, where $\alpha \in (0, N)$
and $N \geq 3$.

\smallskip
A fundamental challenge in analyzing problems like \eqref{mainEq}, \eqref{mainEqCrit}, \eqref{mainEq2} arises from the loss of compactness, which persists even in the subcritical range. This is due to the fact that the Sobolev embedding $H^1(\mathbb{B}^N) \hookrightarrow L^p(\mathbb{B}^N)$ is non-compact for all $2 \leq p \leq \frac{2N}{N-2}$, a consequence of the infinite volume of $\mathbb{B}^N$ and the invariance under the hyperbolic isometry group. Unlike the Euclidean case, where dilation techniques can mitigate such noncompactness, such methods are ineffective in the hyperbolic setting because the conformal and isometric groups coincide \cite{RatcliffeBook}. Subcritical problems are, in general, more tractable in this framework, as the loss of compactness typically arises only along one profile, namely due to hyperbolic translations. In contrast, problems involving the critical Sobolev exponent exhibit more severe compactness issues: even local compactness fails. This phenomenon was rigorously analyzed by Bhakta and Sandeep in \cite{MousomiUTF}, who developed a Palais–Smale decomposition that revealed the concentration occurs along two distinct profiles- one corresponding to the hyperbolic bubble, i.e., solutions of \eqref{hsm} and the other along the localized Aubin-Talenti bubble, i.e., positive solutions of
\begin{equation*}
-\De V =  \abs{V}^{2^*-2}V,  \quad V\in D^{1,2}(\mathbb{R}^N). 
\end{equation*}

\smallskip

The problems \eqref{mainEq} and \eqref{mainEq2} can be analyzed using variational methods. To establish the existence of finite energy solutions, it is essential to examine the associated energy functionals $I$ and $J_\mu$, introduced in \eqref{EnergyFuncSubcritical} and \eqref{energyFunctionalCritical}, respectively. We now state the main results of this article.

Concerning the subcritical equation \eqref{mainEq}, we obtain the following existence and symmetry result for the set $\Lambda$ defined in \eqref{nehariManifoldSubcritical}:

\begin{theorem}[Existence and symmetry]\label{ExistCritPt}
Let $N\ge3$, $s\in(0,1)$, $1<p<2^*-1$ and $\lambda < \frac{(N-1)^2}{4}$. Then,
\begin{enumerate}
\item there exists a minimizer $u$ of $I$ subject to $\Lambda$, i.e. $u \in \Lambda$ and
$I(u)=c^*$. Every such minimizer has constant sign, so there is a nonnegative
minimizer,
\item every minimizer of $I$ subject to $\Lambda$ is a mountain pass weak solution of
\eqref{mainEq},
\item every minimizer of $I$ subject to $\Lambda$ is radially symmetric up to hyperbolic
translations.
\end{enumerate}
\end{theorem}

\noindent
\textbf{Main Novelty and Strategy of the Proofs.} Our analysis is grounded in the variational approach within the natural functional space $H^1(\mathbb{B}^N)$, motivated by the fact that this space embeds continuously into the fractional Sobolev space $H^s(\mathbb{B}^N)$ (see Theorem~\ref{1tosTheorem}). The space $H^s(\mathbb{B}^N)$, associated with the fractional hyperbolic Laplacian, admits various characterizations in the literature, including those based on Fourier analysis. In our work, we adopt the P.V. definition involving a well-structured radial kernel, as developed by Banica et al.~\cite{BanicaEtalFractionalLaplacianHyperbolic}. Leveraging the properties of this radially decreasing kernel (see Lemma~\ref{RadDecreasingLemma}), we establish the continuous embedding $H^1(\mathbb{B}^N) \hookrightarrow H^s(\mathbb{B}^N)$ in Theorem~\ref{1tosTheorem}, which, to the best of our knowledge, has not been previously documented in the literature. This embedding plays a central role in addressing the analytical difficulties introduced by the nonlocal term in problems \eqref{mainEq} and \eqref{mainEq2}. 

To overcome the loss of compactness in the subcritical regime, it is enought to restrict our
attention to radially symmetric functions and work within the subspace
$H^1_r(\mathbb{B}^N)$, where compactness is recovered (see Lemma~\ref{radial Lemma}).
The key tool for reducing general functions to radial ones is Schwarz symmetrization,
which preserves or improves the relevant functional inequalities.
However, in the critical regime, compactness
is also lost along concentrating Aubin--Talenti bubbles, and as in \cite{BrezisNirenberg1983} it
can be recovered only below the energy level determined by the best Sobolev constant of the involved
operator. Understanding this constant is therefore the first step towards studying the critical
problems.

\smallskip
\noindent\textbf{The critical problem and the best Sobolev constant.}
The critical problems \eqref{mainEqCrit} and \eqref{mainEq2} are governed by the mixed
best Sobolev constant
\begin{equation} \label{eq:SK}
    S_{\lambda,s} := \inf_{v \in H^1(\bn) \setminus \{0\}} S_{\lambda,s}(v),
\end{equation}
where, for each $ v \in H^1(\bn) \setminus \{0\} $,
\begin{small}
\begin{align*}
   & S_{\lambda,s}(v):= \notag \\
    &\frac{\displaystyle \int_{\bn}|\nabla_{\bn} v|^2\dvg\hspace{-0.15em}- \hspace{-0.15em}\lambda \int_{\bn}|v|^2\dvg\hspace{-0.2em}+ \frac{1}{2}\int_{\bn}\int_{\bn} |v(x) - v(y)|^2 {\mathcal{K}_s}(d(x,y)) \dvg(x)\dvg(y)}
    { \left( \displaystyle \int_{\bn} |v(x)|^{2^*} \dvg(x) \right)^{2/2^*}}.
\end{align*}
\end{small}
In the Euclidean setting, for the mixed operator $-\Delta+(-\Delta)^s$ the best constant
equals the Sobolev constant $S$ of \eqref{eq:bestSobolevConstantEuclidean} and is never
attained \cite{biagi}. Under the dilation $u_\tau=\tau^{(N-2)/2}u(\tau\,\cdot)$ the
Dirichlet energy and the critical norm are invariant, while
$[u_\tau]_s^2=\tau^{-(2-2s)}[u]_s^2\to0$ as $\tau\to\infty$, so the nonlocal term cannot
lower the constant. On $\bn$ no dilations are available, and the purely local Poincar\'e--Sobolev constant
\begin{equation*}
   S_\lambda^{\mathrm{loc}}
   :=\inf_{u\in H^1(\bn)\setminus\{0\}}
     \frac{\displaystyle\int_{\bn}\bigl(|\nabla_{\bn}u|^2-\lambda u^2\bigr)\,dV_{\bn}}
          {\|u\|_{\cstar}^2}
\end{equation*}
may lie strictly below $S$. This happens when $N\ge4$ and
$\frac{N(N-2)}{4}<\lambda<\frac{(N-1)^2}{4}$ \cite[Proposition~3.2]{DebdipBhaktaKarmakar}.
Our first main result on the critical regime shows that the nonlocal term removes this
gap.
 
\begin{theorem}[Comparison of the mixed, local, and Euclidean best constants]
\label{thm:sobolevConstants}
Let $N\ge3$, $s\in(0,1)$, and $\lambda<\frac{(N-1)^2}{4}$. Then
\begin{equation}\label{eq:sandwich}
   S_\lambda^{\mathrm{loc}}\;\le\;S_{\lambda,s}\;=\;S .
\end{equation}
More precisely,
\begin{enumerate}
\item[\textup{(i)}] $\Smix=S$, and $\Smix$ is not achieved in $\Hh$.
\item[\textup{(ii)}] If $\lambda\le\dfrac{N(N-2)}{4}$, then $\Sloc=\Smix=S$.
\item[\textup{(iii)}] If $N\ge4$ and $\dfrac{N(N-2)}{4}<\lambda<\dfrac{(N-1)^2}{4}$, then
$\Sloc<\Smix$.
\end{enumerate}
\end{theorem}

 When $\lambda\le\frac{N(N-2)}{4}$, the mass coefficient after the conformal change
$\bigl(\frac{N(N-2)}{4}-\lambda\bigr)\int_{\bn}u^2\,\dvg$ is nonnegative and can only raise the
quotient, so $\Sloc\ge S$ while along a concentrating family of Aubin-Talenti bubbles the
mass $\int_{\bn}u^2\,\dvg\to0$, giving $\Sloc\le S$. Hence $\Sloc=\Smix=S$. When
$\frac{N(N-2)}{4}<\lambda<\frac{(N-1)^2}{4}$ and $N\ge4$, the coefficient changes sign and the
strict inequality $\Sloc<S$ bars the concentration phenomenon, i.e., no Aubin--Talenti bubble profile can form
\textup{(}\cite{MousomiUTF}; see also \cite[Lemma~4.2]{DebdipBhaktaKarmakar}\textup{)}.  Therefore, in the range of \textup{(iii)}, $\Sloc$ has an extremal, the positive solution
$W$ of $-\Delta_{\bn}W-\lambda W=W^{\cstar-1}$ \cite{MS}, and adding $(-\Delta_{\bn})^s$
raises the constant back to $S$. 

As a consequence, we obtain a nonexistence result for the purely critical problem.

 \begin{theorem}\label{thm:nonexistCrit}
Let $N\ge3$, $s\in(0,1)$, $\lambda<\frac{(N-1)^2}{4}$, and
\begin{equation*}
   J_0(u)=\frac12\Bigl(\|u\|_\lambda^2+[u]_s^2\Bigr)-\frac1{\cstar}\int_{\bn}|u|^{\cstar}\dvg,
   \quad u\in H^1(\bn),
\end{equation*}
be the energy functional of \eqref{mainEqCrit}, and
\begin{equation*}
   \mathcal N_0:=\Bigl\{u\in H^1(\bn)\setminus\{0\}:\ \|u\|_\lambda^2+[u]_s^2
   =\int_{\bn}|u|^{\cstar}\dvg\Bigr\}
\end{equation*}
its Nehari manifold. Then
\begin{enumerate}
\item[\textup{(i)}] the problem \eqref{mainEqCrit} has no nontrivial weak solution in the ball
\begin{equation*}
   \mathcal B:=\Bigl\{u\in H^1(\bn):\ \|u\|_{L^{\cstar}(\bn)}\le S^{\frac{N-2}{4}}\Bigr\}.
\end{equation*}
Equivalently, every nontrivial weak solution $u$ of \eqref{mainEqCrit} satisfies
$J_0(u)>\frac1NS^{N/2}$.
\item[\textup{(ii)}] \eqref{mainEqCrit} has no ground state, i.e.,
\begin{equation*}
   \inf_{u\in\mathcal N_0}J_0(u)=\frac1NS^{\frac N2}
   \text{ is not attained }
\end{equation*}
\end{enumerate}
\end{theorem}

For $N\ge4$ and $\frac{N(N-2)}{4}<\lambda<\frac{(N-1)^2}{4}$ this is in contrast with the
local problem, whose ground state $W$ has energy $\frac1N(\Sloc)^{N/2}<\frac1NS^{N/2}$.
The nonlocal term destroys it.

\medskip
\noindent\textbf{Equivalence of the two definitions and applications.}
On $\bn$, the fractional Laplacian has a spectral description \eqref{FourierTranformDefn}
and a pointwise description \eqref{PVdefintion}, and the two are used interchangeably in
the literature \cite{BanicaEtalFractionalLaplacianHyperbolic,
harnackinequalityfractionallaplaciantype, fractionalplaplacianhyperbolicspaces}. To the
best of our knowledge, an explicit proof of their equivalence on $\bn$ does not appear in
the literature. The equivalence on closed manifolds in \cite{fracLaplClosedManifolds} does
not cover $\bn$. For completeness, we prove in Proposition~\ref{prop:PVspectralEquiv} that
the two descriptions agree. The proof uses the heat semigroup, and a key geometric input
is the stochastic completeness of $\bn$ \cite{GrigoryanSurvey}. On a manifold where
this fails, the same computation would produce an extra term, and the two descriptions might not agree.
 
The equivalence gives the following bounds (Lemma~\ref{lem:interp}) for $u \in H^1(\bn)$ 
\begin{small}
\begin{equation*}
   \Bigl(\frac{(N-1)^2}{4}\Bigr)^{s}\|u\|_{L^2(\bn)}^2
   \le[u]_s^2
   \le\|u\|_{L^2(\bn)}^{2(1-s)}\,\|\nabla_{\bn} u\|_{L^2(\bn)}^{2s}
   \le\Bigl(\frac{(N-1)^2}{4}\Bigr)^{s-1}\|\nabla_{\bn} u\|_{L^2(\bn)}^2 .
\end{equation*}
\end{small}
The leftmost inequality is what gives $S_{\lambda,s}\ge S$ and the non-attainment in
\eqref{eq:sandwich}. The last inequality is the embedding \eqref{1tosEmbedding} of Theorem~\ref{1tosTheorem},
now with an explicit constant using the equivalence.  The middle inequality is an interpolation inequality. 
 \begin{remark}
In the proof of Theorem~\ref{thm:sobolevConstants} Step 1,
the above inequality together with \eqref{L2FinalEst} gives 
\begin{equation*}
   [u_\varepsilon]_s^2\le\|u_\varepsilon\|_{L^2(\bn)}^{2(1-s)}
   \|\nabla_{\bn}u_\varepsilon\|_{L^2(\bn)}^{2s}
   \le C\bigl(\varepsilon^{\min\{2,N-2\}}\mathcal M_N(\varepsilon)\bigr)^{1-s}\to0,
\end{equation*}
since $\|\nabla_{\bn}u_\varepsilon\|_{L^2(\bn)}$ is bounded. For proving $\Smix\le S$ only
$[u_\varepsilon]_s\to0$ matters. We nevertheless keep Step~1, for two reasons.
First, it uses only the kernel definition \eqref{PVdefintion}, so the upper bound $\Smix\le S$ does
not depend on Proposition~\ref{prop:PVspectralEquiv}, and the blow-up estimates for the nonlocal term
along the concentrating bubbles are also of independent interest. Second, the rate is crucial
in the verification of the compactness condition \eqref{CompactnessAssumptionOnFunctional},
where the nonlocal cost must be smaller than the gain $\varepsilon^{\beta}$ from the
perturbation, and this gives the existence results in Theorem~\ref{thm:perturbMu}.
\end{remark}
The interpolation inequality is also the ingredient that makes a Struwe-type
decomposition of Palais--Smale sequences possible for the mixed operator. In the
decomposition of \cite{MousomiUTF} for the local problem, the loss of compactness occurs
along hyperbolic translations and along Euclidean bubbles concentrating at a point. The
seminorm $[\,\cdot\,]_s$ is invariant under the isometries of $\bn$, because the kernel
depends only on $d(x,y)$. A
concentrating bubble at scale $\varepsilon\to0$ stays bounded in $H^1(\bn)$ while its
$L^2$-norm tends to zero. The interpolation inequality then forces its nonlocal energy to
vanish, and the limiting Euclidean profile solves $-\Delta V=|V|^{\cstar-2}V$ in $\RN$. The embedding $H^1(\bn)\hookrightarrow
H^s(\bn)$ alone only bounds $[\,\cdot\,]_s$ by $\|\nabla_{\bn}\,\cdot\,\|_{L^2}$, which does
not vanish along the bubble. We do not pursue the full decomposition here. For a global
compactness result for the mixed operator in the Euclidean setting, see \cite{Souptiketal}.

\medskip
\noindent\textbf{The perturbed critical problem.} We finally turn to \eqref{mainEq2}.
By Theorem~\ref{thm:nonexistCrit} the purely critical problem has no solution at or below
the first energy level, and a positive perturbation $\mu|u|^{p-1}u$ is what allows one to
go below it. We first prove the following existence result under a compactness assumption.
 
\begin{theorem}\label{CriticalExponentMainThm}
Let $N\ge3$, $s\in(0,1)$, $\lambda<\frac{(N-1)^2}{4}$, $ 1 < p < 2^* - 1 $ and $\mu>0$, and
suppose there exists a nonnegative radial function $ u_0 \in H_r^1(\bn) \setminus \{0\} $
such that
\begin{equation}\label{CompactnessAssumptionOnFunctional}
    \sup_{\zeta \geq 0} J_\mu(\zeta u_0) < \frac{1}{N} S_{\lambda,s}^{N/2},
\end{equation}
where $J_\mu$ is the energy functional \eqref{energyFunctionalCritical} of
\eqref{mainEq2} and $S_{\lambda,s}$ is defined in \eqref{eq:SK}. Then problem
\eqref{mainEq2} admits a nontrivial nonnegative radial solution $ u \in H_r^1(\bn) $.
\end{theorem}
 
The noncompactness issues inherent in the hyperbolic setting are managed through carefully crafted variational techniques. 
For the problem
\eqref{mainEq2}, involving critical nonlinearities, the main difficulty lies in the
failure of the Palais-Smale condition for the associated energy functional $J$, and
this is precisely where the condition \eqref{CompactnessAssumptionOnFunctional} comes
into play. The threshold condition \eqref{CompactnessAssumptionOnFunctional} is not a
technical artifact but the exact analytic manifestation of the loss of compactness in
the critical regime, playing here the role that the strict sub-threshold estimate
plays in the classical Brezis-Nirenberg analysis~\cite{BrezisNirenberg1983} and in its fractional
counterpart~\cite{servadei2015brezis}. Under this condition, we establish the existence
of a nontrivial solution through a general mountain-pass framework.
 
The condition \eqref{CompactnessAssumptionOnFunctional} can be verified by using the
equality $\Smix=S$ of Theorem~\ref{thm:sobolevConstants}\textup{(i)} and testing
$J_\mu$ with the concentrating family $u_\varepsilon$ of
\eqref{eq:concAubinTalentiBubbles}. As in \cite[Theorem~1.5]{biagi}, the outcome depends
on the competition between the cost of the nonlocal term along these bubbles of order
$\varepsilon^{\sigma}$ with
\begin{equation*}
   \sigma=\begin{cases}2-2s,&N+2s>4,\\1-s,&N=3,\ s\le\frac12,\end{cases}
\end{equation*}
and the gain from the perturbation of order $\varepsilon^{\beta}$ with
$\beta=N-\frac{(N-2)(p+1)}{2}$. The two cases in $\sigma$ come from the Aubin--Talenti bubble $U$ itself. We show in
Lemma~\ref{lem:Ufinite} that its Gagliardo seminorm $[U]_{s,\RN}$ is finite if and only
if $N+2s>4$. For $N\le4$, this is not covered by the argument in \cite[Remark~3.2]{biagi}.
 
\begin{theorem}\label{thm:perturbMu}
Let $N\ge3$, $s\in(0,1)$, $\lambda<\frac{(N-1)^2}{4}$ and $1<p<2^*-1$.
\begin{enumerate}
\item[\textup{(i)}] If $\beta<\sigma$, that is, if
\begin{equation}\label{eq:prange}
   1+\frac{4s}{N-2}<p<\frac{N+2}{N-2}\quad\text{if }N+2s>4,
   \qquad
   3+2s<p<5\quad\text{if }N=3,\ s\le\tfrac12 ,
\end{equation}
then problem \eqref{mainEq2} admits a nontrivial nonnegative radial solution for every $\mu>0$.
\item[\textup{(ii)}] For every $p\in(1,2^*-1)$ there exists $\mu_0>0$ such that problem
\eqref{mainEq2} admits a nontrivial nonnegative radial solution for every $\mu>\mu_0$.
\item[\textup{(iii)}] If $\mu\le0$, problem \eqref{mainEq2} has no nontrivial weak solution
in the ball $\mathcal B$ of Theorem~\ref{thm:nonexistCrit}.
\end{enumerate}
\end{theorem}

\textit{The article is organized as follows}: In Section~\ref{PrelimSection}, we introduce
the geometric preliminaries and notational conventions, prove the continuous embedding
$H^1(\mathbb{B}^N) \hookrightarrow H^s(\mathbb{B}^N)$, prove the equivalence of the two
definitions of the fractional Laplacian together with the interpolation inequality
(Proposition~\ref{prop:PVspectralEquiv} and Lemma~\ref{lem:interp}), establish the
comparison of the best Sobolev constants (Theorem~\ref{thm:sobolevConstants}), and
prove Theorem~\ref{thm:nonexistCrit}.
Section~\ref{subcrticalSection} is devoted to the proof of Theorem~\ref{ExistCritPt} for the subcritical case, while Section~\ref{criticalSection} is devoted to the perturbed critical problem
\eqref{mainEq2} and contains the proofs of Theorems~\ref{CriticalExponentMainThm}
and~\ref{thm:perturbMu}. Finally, Section~\ref{appendixSection} provides auxiliary results, including the weak maximum principle. Throughout, $C$ denotes a positive constant whose value may change from line to line. It may depend on $N$, $s$, $p$, $\lambda$, $\mu$ and on fixed auxiliary quantities, but
never on $\varepsilon$, on the index of a sequence, or on the functions being estimated.
Where the dependence matters, it is indicated, as in $C(N,s,r)$.

\section{Preliminaries}\label{PrelimSection}
In this section, we will introduce some of the notations and definitions used in this
paper and also recall some of the embeddings
related to the Sobolev space on the hyperbolic space. 
\smallskip
\noindent
We will denote by $\bn$ the disc model of the hyperbolic space, i.e., the Euclidean unit ball $B(0,1):= \{x \in \mathbb{R}^N: |x|^2<1\}$ equipped with the Riemannian metric
\begin{align*}
	{\rm d}s^2 = \left(\frac{2}{1-|x|^2}\right)^2 \, {\rm d}x^2,
\end{align*}
where ${\rm d}x$ is the standard Euclidean metric and $|x|^2 = \sum_{i=1}^Nx_i^2$ is the standard Euclidean length. The corresponding volume element is given by $\mathrm{~d} V_{\mathbb{B}^{N}} = \big(\frac{2}{1-|x|^2}\big)^N {\rm d}x, $ where ${\rm d}x$ denotes the Lebesgue 
measure on $\mathbb{R}^N$. $\nabla_{\bn}$ and $\Delta_{\bn}$ denote gradient 
 vector field and Laplace-Beltrami operator, respectively, which take the form
\begin{align*} 
 \nabla_{\bn} = \left(\frac{1 - |x|^2}{2}\right)^2\nabla,  \quad 
 \Delta_{\bn} = \left(\frac{1 - |x|^2}{2}\right)^2 \Delta + (N - 2)\left(\frac{1 - |x|^2}{2}\right)  x \cdot \nabla,
\end{align*}
where $\nabla, \Delta$ are the standard Euclidean gradient vector field and Laplace operator, respectively, and '$\cdot$' denotes the 
standard inner product in $\mathbb{R}^N.$

 \smallskip 

\noindent

\subsection{M\"{o}bius Transformations and Convolution} \cite{LP}
\noindent
The M\"{o}bius transformations $T_a$ for each $a \in \bn$ are defined as follows: 
\begin{equation*}
T_a(x)=\frac{|x-a|^2 a-\left(1-|a|^2\right)(x-a)}{1-2 x \cdot a+|x|^2|a|^2} 
\end{equation*}
where $x \cdot a=x_1 a_1+x_2 a_2+\cdots+x_n a_n$ represents the scalar product in $\mathbb{R}^N$. The measure on $\bn$ is known to be invariant under M\"{o}bius transformations.
\subsection{Hyperbolic Distance on $\bn$} The distance between $x$ and $y$ in $\bn$ can be defined as follows utilizing the M\"obius transformations:
\begin{equation*}
\rho(x, y)=\rho\left(T_x(y)\right)=\rho\left(T_y(x)\right)=\log \frac{1+\left|T_y(x)\right|}{1-\left|T_y(x)\right|}
\end{equation*}
where $\rho(x)= d(x,0)= \log \frac{1+|x|}{1-|x|}$ is the geodesic distance from the origin.\\
As a result, a subset of $\bn$ is a hyperbolic sphere in $\bn$ if and only if it is a Euclidean sphere in $\mathbb{R}^N$ and contained in $\bn$, possibly 
with a different centre and different radius, which can be explicitly computed from the formula of $d(x,y)$ \cite{RatcliffeBook}.  Geodesic balls in $\bn$ of radius $r$ centred at $x\in\bn$ will be denoted by
$$
   B(x,r):=\{\,y\in\bn : d(x,y)<r\,\},
$$
while Euclidean balls in $\RN$ of radius $r$ centred at $x\in\RN$ will be denoted by
$$
   B_{\RN}(x,r):=\{\,y\in\RN : |x-y|<r\,\}.
$$
This section introduces the mathematical tools and definitions necessary for our analysis.

\subsection{Helgason Fourier Transform on the Hyperbolic Space}\label{sec:fourier transform prelims}
We give a brief overview of the Helgason Fourier transform, which serves as a foundation for defining the fractional Laplacian in the hyperbolic space. For comprehensive treatments and further references, see \cite{BanicaEtalFractionalLaplacianHyperbolic, LP}.\\
Analogous to the Euclidean context, the Fourier transform can be defined as follows,
\begin{equation*}
\hat{f}(\beta, \theta)=\int_{\mathbb{B}^N} f(x) h_{\beta, \theta}(x) \dvg,
\end{equation*}
for $\beta \in \mathbb{R}, \theta \in \mathbb{S}^{N-1}$, where $h_{\beta, \theta}$ are the generalized eigenfunctions of the Laplace Beltrami operator that satisfy
\begin{equation*}
\Delta_{\mathbb{B}^N} h_{\beta, \theta}=-\left(\beta^2+\frac{(N-1)^2}{4}\right) h_{\beta, \theta} .
\end{equation*}
Furthermore, given $f \in C_c^\infty(\bn)$, the following inversion formula is valid:
\begin{equation*}
f(x)=\int_{-\infty}^{\infty} \int_{\mathbb{S}^{N-1}} \bar{h}_{\beta, \theta}(x) \hat{f}(\beta, \theta) \frac{\mathrm{d} \beta\mathrm{d} \theta}{|c(\beta)|^2},
\end{equation*}
where  $c(\beta)$ represents the Harish-Chandra coefficient:
\begin{equation*}
\frac{1}{|c(\beta)|^2}=\frac{1}{2} \frac{\left|\Gamma\left(\frac{N-1}{2}\right)\right|^2}{|\Gamma(N-1)|^2} \frac{\left\lvert\, \Gamma\left(i \beta+\left.\left(\frac{N-1}{2}\right)\right|^2\right.\right.}{|\Gamma(i \beta)|^2} .
\end{equation*} 
Moreover, the following Plancherel formula holds
\begin{equation}
\int_{\mathbb{B}^N}|f(x)|^2 \dvg=\int_{\mathbb{R} \times \mathbb{S}^{N-1}}|\hat{f}(\beta, \theta)|^2 \frac{\mathrm{d} \beta\mathrm{d} \theta}{|c(\beta)|^2}. \label{plancheralFormulaEqn}
\end{equation}
\smallskip
\noindent\textit{Parseval's identity.}
For $h,g\in L^2(\bn)$,
\begin{equation}\label{eq:parseval}
   \int_{\bn}h\,\overline g\,\dvg
   =\int_{\R\times\Sph}\hat h\,\overline{\hat g}\,
     \frac{\mathrm{d}\beta\,\mathrm{d}\theta}{|c(\beta)|^2}.
\end{equation}
\noindent
It is straightforward to verify through integration by parts for compactly supported functions, and thus for every $f \in L^2(\bn)$ that
\begin{equation}
\begin{aligned}
\widehat{\Delta_{\mathbb{B}^N} f}(\beta, \theta) & =\int_{\mathbb{B}^N} \Delta_{\mathbb{B}^N} f(x) h_{\beta, \theta}(x) \dvg(x) =\int_{\mathbb{B}^N} f(x) \Delta_{\mathbb{B}^N} h_{\beta, \theta}(x) \dvg(x) \\
& =-\left(\beta^2+\frac{(N-1)^2}{4}\right) \hat{f}(\beta, \theta) .
\end{aligned} \label{eq: laplacian fourier transfom}
\end{equation}

\noindent
Given the theory discussed above, the fractional Laplacian on the hyperbolic space can be defined as  ${\left(-\Delta_{\mathbb{B}^N}\right)}^{s} f$, which is the operator satisfying
\begin{equation}
(-\widehat{\left.\Delta_{\mathbb{B}^N}\right)^s} f=\left(\beta^2+\frac{(N-1)^2}{4}\right)^{s} \hat{f}, \quad s \in \mathbb R. \label{FourierTranformDefn}
\end{equation}
For $N \geq 2$ and $s \in (0,1)$, the fractional Laplacian admits the following pointwise representation:

\begin{equation}
\left(-\Delta_{\mathbb{B}^N}\right)^s u(x)= \text { P.V. } \int_{\mathbb{B}^N}(u(x)-u(\xi)) \mathcal{K}_{s}(d(x, \xi)) \mathrm{d} \xi \label{PVdefintion}
\end{equation}
where $\rho := d(x, \xi)$ denotes the hyperbolic distance between the points $x$ and $\xi$, and the kernel $\mathcal{K}_s(\rho)$ is given explicitly as follows (see \cite[Proposition 1.2]{fractionalplaplacianhyperbolicspaces}, \cite{BanicaEtalFractionalLaplacianHyperbolic}):\\

\noindent
when $N \geq 3$ is odd 
\begin{equation*}
\mathcal{K}_{s}(\rho)=C(N,s)\left(\frac{-\partial_\rho}{\sinh \rho}\right)^{\frac{N-1}{2}}\left(\rho^{-\frac{1+2s }{2}} K_{\frac{1+2s }{2}}\left(\frac{N-1}{2} \rho\right)\right),
\end{equation*}

\noindent
and when $N \geq 2$ is even
\begin{equation*}
\mathcal{K}_{s}(\rho)=C(N,s) \int_\rho^{\infty} \frac{\sinh r}{\sqrt{\pi} \sqrt{\cosh r-\cosh \rho}}\left(\frac{-\partial_r}{\sinh r}\right)^{\frac{N}{2}}\left(r^{-\frac{1+2s}{2}} K_{\frac{1+2s}{2}}\left(\frac{N-1}{2} r\right)\right) \mathrm{d}r,
\end{equation*}
\noindent
where $K_\nu$ denotes the modified Bessel function of the second kind, and the constant $C(N, s)$ is given by
\noindent
\begin{equation}
C(N, s)= \frac{\sqrt{\pi}\;2^{2 s}\;}{2\Gamma(\frac{3}{2})\pi^{\frac{N}{2}}|\Gamma(-s)|}\frac{1}{2^{\frac{N-2+2s}{2}} }\left(\frac{N-1}{2}\right)^{\frac{1+2s}{2}}. \label{fractionalPVdefnConstant}
\end{equation}
In the notation of \cite{fractionalplaplacianhyperbolicspaces},
$C(N,s)=c_{N,s,2}\,C_1$.

\medskip
\noindent\subsection{Semigroup approach.}
We use the framework of
\cite[\S2.5]{gupta2024existencesymmetryregularityground}. For
$f\in C_c^\infty(\bn)$ let $u=u(x,t)$ solve the heat equation
$\partial_tu=\lap u$ on $\bn\times\R^+$ with $u(x,0)=f(x)$. Applying the Fourier
transform in $x$ and using Section \ref{sec:fourier transform prelims} gives
\begin{equation}\label{eq:heatdef}
   \hat u(\beta,\theta,t)
   =e^{-t\left(\beta^2+\frac{(N-1)^2}{4}\right)}\hat f(\beta,\theta)
   =\widehat{e^{t\lap}f}(\beta,\theta),
\end{equation}
where $f\mapsto e^{t\lap}f$ denotes the solution operator, and
\begin{equation}\label{eq:heatkernel}
   e^{t\lap}f(x)=p_{t,N}*f(x)
   =\int_{\bn}p_{t,N}\bigl(d(x,\xi)\bigr)f(\xi)\,\dvg(\xi),
\end{equation}
where $p_{t,N}(x,\xi)=p_{t,N}(d(x,\xi))$ is the heat kernel on $\bn$.

\smallskip
Of the heat kernel, we use that it is positive, that it is symmetric in its
two arguments, $d$ being the geodesic distance, and that $\bn$ is stochastically
complete, i.e.,
\begin{equation}\label{eq:stochcomplete}
   \int_{\bn}p_{t,N}\bigl(d(x,\xi)\bigr)\,\dvg(\xi)=1
   \qquad\text{for every }t>0\text{ and every }x\in\bn ,
\end{equation}
which holds by the volume test \cite[Theorem~9.1]{GrigoryanSurvey}. Indeed
$\bn$ is complete, and a geodesic ball of radius
$r\ge1$ has volume
$V(r)=|\Sph|\int_0^r\sinh^{N-1}\rho\,\mathrm{d}\rho\le|\Sph|\,r\,e^{(N-1)r}$, so
that $\log V(r)\le Cr$ and $\int_1^\infty\frac{r\,\mathrm{d}r}{\log V(r)}=\infty$.

\medskip
\subsection{ A sharp Poincar\'{e}-Sobolev inequality} (see \cite{MS})

\medskip
\noindent
{\bf Sobolev Space:} We will denote by ${H^{1}}(\bn)$ the Sobolev space on the disc
model of the hyperbolic space $\bn$, equipped with norm $\|u\|=\left(\int_{\mathbb{B}^N} |\nabla_{\mathbb{B}^{N}} u|^{2}\right)^{\frac{1}{2}},$
where  $|\nabla_{\bn} u| $ is given by
$|\nabla_{\bn} u| := \langle \nabla_{\bn} u, \nabla_{\bn} u \rangle^{\frac{1}{2}}_{\bn} .$ 

\noindent
For $N \geq 3$ and every $p \in \left(1, \frac{N+2}{N-2} \right]$ there exists an optimal constant 
$S_{\lambda,p} > 0$ such that
\begin{equation}
	S_{\lambda,p} \left( \int_{\mathbb{B}^{N}} |u|^{p + 1} \mathrm{~d} V_{\mathbb{B}^{N}} \right)^{\frac{2}{p + 1}} 
	\leq \int_{\mathbb{B}^N} \left[|\nabla_{\mathbb{B}^{N}} u|^{2}
	- \frac{(N-1)^2}{4} u^{2}\right] \, \mathrm{~d} V_{\mathbb{B}^{N}}, \label{PoinSobIneq}
\end{equation}
for every $u \in C^{\infty}_{0}(\mathbb{B}^{N}).$ If $ N = 2$, then any $p > 1$ is allowed.

\noindent
 A crucial point to note is that the bottom of the spectrum of $- \Delta_{\bn}$ on $\bn$ is 
\begin{equation}\label{firsteigen}
	\frac{(N-1)^2}{4} = \inf_{u \in H^{1}(\bn)\setminus \{ 0 \}} 
	\dfrac{\int_{\bn}|\nabla_{\bn} u|^2 \, \mathrm{~d} V_{\mathbb{B}^{N}} }{\int_{\bn} |u|^2 \, \mathrm{~d} V_{\mathbb{B}^{N}}}. 
\end{equation}

\begin{remark}
	As a result of \eqref{firsteigen}, if $\lambda < \frac{(N-1)^2}{4},$ then 
\begin{equation*}
 \left\|u\right\|_{\lambda} := \left[ \int_{\bn} \left( |\nabla_{\bn} u|^2 - \lambda \, u^2 \right) \, \mathrm{~d} V_{\mathbb{B}^{N}} \right]^{\frac{1}{2}}, \quad u \in C_c^{\infty}(\bn)
\end{equation*}
	is a norm, equivalent to the $H^1(\bn)$ norm. When $\lambda = \frac{(N-1)^2}{4}$, the sharp Poincar\'e inequality \eqref{PoinSobIneq} ensures that $\|u\|_{\frac{(N-1)^2}{4}}$ is also a norm on $C_c^{\infty}(\bn)$.
\end{remark}
\noindent
For $\lambda \leq \frac{(N-1)^2}{4}$, denote by $\mathcal{H}_\lambda(\bn)$ the completion of $C_c^{\infty}(\bn)$ with respect to the norm $\|u\|_\lambda$. The associated inner product is denoted by $\langle \cdot, \cdot\rangle_{{\lambda}}.$ Note that
\begin{equation}
S_{\lambda, p}\left(\int|u|^{p+1} \, d V_{\bn}\right)^{\frac{2}{p+1}} \leq \|u\|^2_\lambda, \quad p \in\left(1, \frac{N+2}{N-2}\right] \quad \text{for} \quad u \in \mathcal{H}_\lambda(\bn). \label{SubcritPoincare}
\end{equation}
Additionally, throughout this article, $\|\cdot\|_{r}$ denotes the $L^r$-norm on
$\BN$ with respect to the volume measure $dV_{\BN}$, and $\|\cdot\|_{L^r(\RN)}$ the
$L^r$-norm on $\RN$ with respect to the Lebesgue measure, for $1\le r\le\infty$.
\medskip

\subsection{Fractional Sobolev Spaces}

There exist multiple definitions of Sobolev spaces on the hyperbolic space, depending on the underlying analytical framework. In this work, we shall employ the following definition for the fractional Sobolev space $H^s(\bn)$ with $0 < s < 1$:
\begin{equation}
    H^s(\bn) = \left\{ f \in L^2(\bn) \mid \|f\|^2_{L^2(\bn)} + \left[f\right]_s^2< \infty \right\}. \label{eq:spectral fractional}
\end{equation}
where the seminorm $[f]_s$ is defined by

$$
[f]_s := \left( \frac{1}{2}\int_{\bn} \int_{\bn} |f(x) - f(y)|^2
         \mathcal{K}_s(d(x,y)) \dvg(x) \dvg(y)\right)^{\frac{1}{2}}.
$$
The full norm is then defined as
$$
\|f\|_{H^{s}(\bn)} := \left( \int_{\bn} |f(x)|^2 \dvg(x) + [f]_{s}^2 \right)^{\frac{1}{2}}.
$$
This space serves as an intermediary Banach space between $L^2(\bn)$ and $H^1(\bn)$, endowed with the natural norm defined above. In what follows, we establish the continuous embedding stated in \eqref{1tosEmbedding}. As a key step towards this, we first prove an important auxiliary result concerning the behavior of the kernel $\mathcal{K}_s$.

\begin{lemma}\label{RadDecreasingLemma}
Let $0 < s < 1$. Then the kernel $\mathcal{K}_s$, defined in \eqref{PVdefintion}, is positive and strictly decreasing with respect to the geodesic distance $\rho$.
\end{lemma}

\begin{proof}
The positivity of the kernel $\mathcal{K}_s$ follows from \cite[Corollary 3.3]{fractionalplaplacianhyperbolicspaces}. For notational convenience, define
\begin{equation*}
\mathscr{K}_{\nu,a}(\rho) := \rho^{-\nu} K_\nu(a\rho),
\end{equation*}
where $K_\nu$ is the modified Bessel function of the second kind, with parameters $\nu \in \mathbb{R}$ and $a > 0$.

\medskip
\noindent
\textbf{Case 1: $N \geq 3$ odd.} In this case, the kernel admits the explicit formula
$$
\mathcal{K}_s(\rho) = C(N,s) \left( \frac{-\partial_\rho}{\sinh \rho} \right)^{\frac{N-1}{2}} \mathscr{K}_{\frac{1+2s}{2}, \frac{N-1}{2}}(\rho),
$$
for a constant $C(N,s) > 0$ defined in \eqref{fractionalPVdefnConstant}. Differentiating this expression, we get
\begin{align*}
\frac{d}{d\rho} \mathcal{K}_s(\rho) 
&= C(N,s) \frac{d}{d\rho} \left( \left( \frac{-\partial_\rho}{\sinh \rho} \right)^{\frac{N-1}{2}} \mathscr{K}_{\frac{1+2s}{2}, \frac{N-1}{2}}(\rho) \right) \\
&=C(N,s)(-\sinh \rho)\frac{-1}{\sinh \rho}\frac{\partial}{\partial \rho}\left[-\frac{1}{\sinh \rho} \frac{\partial}{\partial \rho}\right]^{\frac{N-1}{2}}\mathscr{K}_{\frac{1+2s}{2}, \frac{N-1}{2}}(\rho)\\
 &= C(N,s)(-\sinh \rho)\left(\frac{-\partial_\rho}{\sinh \rho}\right)^{\frac{N+1}{2}}\mathscr{K}_{\frac{1+2s}{2}, \frac{N-1}{2}}(\rho)<0,
\end{align*}
where the inequality in the final step follows from \cite[Lemma 3.2]{fractionalplaplacianhyperbolicspaces}.

\medskip
\noindent
\textbf{Case 2: $N \geq 2$ even.} In this case, the kernel is represented by the integral
\begin{align*}
     \mathcal{K}_{s}(\rho)&=C(N,s) \int_\rho^{\infty} \frac{\sinh r}{\sqrt{\pi} \sqrt{\cosh r-\cosh \rho}}\left(\frac{-\partial_r}{\sinh r}\right)^{\frac{N}{2}}\left(r^{-\frac{1+2s}{2}} K_{\frac{1+2s}{2}}\left(\frac{N-1}{2} r\right)\right) \mathrm{d} r\\
     &= C(N,s) \int_\rho^{\infty} \frac{\sinh r}{\sqrt{\pi} \sqrt{\cosh r-\cosh \rho}}\left(\frac{-\partial_r}{\sinh r}\right)^{\frac{N}{2}}\mathscr{K}_{\frac{1+2s}{2},\frac{N-1}{2}}(r)\mathrm{d} r 
     \end{align*}
Now define, for $m \in \mathbb{N} \cup \{0\}$,
$$
F_m(r) := \frac{\sinh r}{\sqrt{\cosh r - \cosh \rho}} \left( \frac{-\partial_r}{\sinh r} \right)^m \mathscr{K}_{\nu,a}(r),
$$
where $\nu >1/2$ and $a \geq 1/2$. With this notation, the kernel becomes
$$
\mathcal{K}_s(\rho) = \frac{C(N,s)}{\sqrt{\pi}} \int_\rho^{\infty} F_{\frac{N}{2}}(r) \, dr,
$$
for $\nu = \frac{1+2s}{2}$ and $a = \frac{N-1}{2}$.
According to \cite[Lemma 3.1]{fractionalplaplacianhyperbolicspaces}, $F_m$ is integrable on $(\rho, \infty)$ and satisfies the recursive identity
$$
\left( \frac{-\partial_\rho}{\sinh \rho} \right) \int_\rho^{\infty} F_m(r) \, dr = \int_\rho^{\infty} F_{m+1}(r) \, dr.
$$
Using this identity, we compute
\begin{align*}
\frac{d}{d\rho} \mathcal{K}_s(\rho)
&= \frac{C(N,s)}{\sqrt{\pi}} \frac{d}{d\rho} \int_\rho^{\infty} F_{\frac{N}{2}}(r) \, dr \\
&= \frac{C(N,s)}{\sqrt{\pi}} (-\sinh \rho) \int_\rho^{\infty} F_{\frac{N}{2}+1}(r) \, dr<0.
\end{align*}
The inequality again follows from \cite[Lemma 3.2]{fractionalplaplacianhyperbolicspaces}, which ensures that the integrand $F_{\frac{N}{2}+1}(r)$ is strictly positive. This confirms that $\mathcal{K}_s$ is strictly decreasing in $\rho$.
\end{proof}

\medskip

\begin{theorem}\label{1tosTheorem}
Let $0 < s < 1$. Then for all $f \in H^1(\bn)$,
$$
\left[f\right]_s^2\lesssim \int_{\bn}|\nabla_{\bn}f|^2\dvg.
$$
In particular, keeping in mind \eqref{firsteigen}, we obtain the following continuous embedding
\begin{equation} \label{1tosEmbedding}
    H^1(\bn) \hookrightarrow H^s(\bn)\;\;\;\text{for } 0 < s < 1.
\end{equation}
\end{theorem}

\begin{proof} We first let $f\in C_c^\infty(\bn)$, and extend to $f\in\Hh$ by density at the end. The kernel $\mathcal{K}_s$ is positive, and the asymptotic behavior of $\mathcal{K}_s$, as described in \cite[Theorem 1.2]{fractionalplaplacianhyperbolicspaces}, \cite{BanicaEtalFractionalLaplacianHyperbolic}, guarantees the existence of constants $l, m > 0$ such that
\begin{equation}
 \mathcal{K}_{s}(\rho) \lesssim \rho^{-N-2s} \;\;\text{for}\;\; \rho < l, \label{eq:nearkernelestimate}
 \end{equation}
 and 
 \begin{equation*}\mathcal{K}_{s}(\rho) \lesssim \rho^{-1-s} e^{-(N-1)\rho} \;\;\text{ for } \rho > m.
 \end{equation*}

\noindent
Now consider
\begin{align*}
    &\int_{\bn}\int_{\bn}(f(x)-f(y))^2 \mathcal{K}_{s}(d(x,y))\dvg(x)\dvg(y)\\
&=\underbrace{\int_{\bn}\int_{B(y,l)}}_{I_1}+\underbrace{\int_{\bn}\int_{B(y,m)^{\complement}}}_{I_2}+ \underbrace{\int_{\bn}\int_{B(y,m)\setminus B(y,l)}}_{I_3}.
\end{align*}
First, consider the integral $ I_1 $, where for each pair $ x, y \in \bn $, we denote by $ \gamma_{x,y} : [0,1] \to \bn $ the minimizing geodesic joining $ x $ to $ y $. Then, we have
\begin{align*}
    I_1&:=\int_{\bn}\int_{B(y,l)}(f(x)-f(y))^2 \mathcal{K}_{s}(d(x,y))\dvg(x)\dvg(y)\\
    &\leq C \int_{\bn}\int_{B(y,l)}\int_0^1 \frac{|\nabla_{\bn} f(\gamma_{x,y}(t))|^2}{(d(x,y))^{N+2s-2}}dt\dvg(x)\dvg(y) 
   \end{align*}

\medskip
\noindent
We now exploit the symmetry of the integral. Observe that $ 0 \leq t \leq 1 $, the identity $ \gamma_{x,y}(t) = \gamma_{y,x}(1-t) $ holds. Using the change of variables $ (x, y, t) \mapsto (y, x, 1 - t) $, we deduce that the integral over $ t \in [0,1] $ can be symmetrically rewritten over $ t \in [\frac{1}{2},1] $, yielding
\begin{align*}
    I_1\leq \;& C\int_{\bn}\int_{B(y,l)}\int_0^1 \frac{|\nabla_{\bn}f(\gamma_{x,y}(t))|^2}{(d(x,y))^{N+2s-2}}dt\dvg(x)\dvg(y) \\
    & = 2C\int_{\bn}\int_{\bn}\int_{\frac{1}{2}}^1 \frac{|\nabla_{\bn}f(\gamma_{x,y}(t))|^2 \mathbf{1}_{B(y, l)}(x)}{(d(x,y))^{N+2s-2}}dt\dvg(x)\dvg(y).
\end{align*}
Now, employing \cite[Theorem 3.3.9]{SC} (see also \cite{SC2, Ch}), we obtain a uniform lower bound on the Jacobian $ J_{x,t}(y) $ of the geodesic map
\begin{equation}
\Phi_{x, t}: y \mapsto \gamma_{x, y}(t) \label{ChangeOfVariables}
\end{equation}
\noindent
valid for all $ x, y \in \bn $ with $ d(x,y) < l $ and all $ t \in [\frac{1}{2},1] $. Specifically, we have
\begin{equation*}
J_{x, t}(y) \geq 1 / F(l), 
\end{equation*} 
where $ F(l) $ is some positive function. Using this lower bound, we estimate
\begin{align*}
    I_1&\leq 2C\int_{\bn}\int_{\bn}\int_{\frac{1}{2}}^1 \frac{|\nabla_{\bn} f(\gamma_{x,y}(t))|^2 \mathbf{1}_{B(y, l)}(x)}{(d(x,y))^{N+2s-2}}dt\dvg(x)\dvg(y) \\
    &\leq 2C F(l)\int_{\bn}\int_{\bn}\int_{\frac{1}{2}}^1 \frac{|\nabla_{\bn} f(\gamma_{x,y}(t))|^2 \mathbf{1}_{B(y, l)}(x)J_{x, t}(y) }{(d(x,y))^{N+2s-2}}dt\dvg(x)\dvg(y)\\
      &\leq C(N,l)\int_{\bn}\int_{\bn}\int_{\frac{1}{2}}^1\frac{|\nabla_{\bn} f(\gamma_{x,y}(t))|^2 \mathbf{1}_{B(\gamma_{x,y}(t), l)}(x)J_{x, t}(y) }{(d(\gamma_{x,y}(t),x))^{N+2s-2}}dt\dvg(x)\dvg(y).
      \end{align*}
      \noindent
In the above series of estimates, in the third step we used the fact that for $ t \in (\frac{1}{2},1) $, the point $ \gamma_{x,y}(t) $ lies on the geodesic segment joining $ x $ and $ y $ then $d(x,y)\geq d(\gamma_{x,y}(t),x) $. Now we can apply Fubini–Tonelli’s theorem and perform the change of variables as indicated in \eqref{ChangeOfVariables}. Letting $ z := \gamma_{x,y}(t) $, the integral becomes
\begin{align*}
      I_1&\leq C(N,l)\int_{\frac{1}{2}}^1\int_{\bn}\int_{\bn}\frac{|\nabla_{\bn} f(z)|^2 \mathbf{1}_{B(z, l)}(x)}{(d(z,x))^{N+2s-2}}\dvg(z)\dvg(x)dt\\
       & \leq C(N,l)\|\nabla_{\bn}f\|_{L^2(\bn)}^2\int_{B(0,l)}\frac{1}{(d(z,0))^{N+2s-2}}\dvg(z)\\
    &= C(N,l)\|\nabla_{\bn} f\|_{L^2(\bn)}^2 \int_{\mathbb{S}^{N-1}}\int_0^l \frac{(\sinh r )^{N-1}}{r^{N+2s-2}}\mathrm{~d}r \mathrm{~d}\sigma.
    \end{align*}
 Finally, to make sense of this last integral, we can choose $l$ small enough so that, for small  $r $ (i.e., $ 0 < r < 1 $), we use  $ \sinh r = r + O(r^3). $ Then,
\begin{equation*}
    (\sinh r)^{N-1} = r^{N-1} + O(r^{N+1}). \end{equation*} 
Hence
\begin{align*}
&\int_0^l \frac{(\sinh r )^{N-1}}{r^{N+2s-2}}\mathrm{~d}r= \int_0^l \frac{r^{N-1} + O(r^{N+1})}{r^{N+2s-2}} \, dr\\
&=\int_0^l r^{-2s+1} dr + O \left( \int_0^l r^{-2s+3} dr \right)=\frac{l^{-2s+2}}{-2s+2}
+O(l^{-2s+4}) \text{ for $s \in (0,1)$}.
\end{align*}
Therefore,
\begin{equation*}I_1 \leq C\|\nabla_{\bn} f\|_{L^2(\bn)}^2.\end{equation*}
\noindent
Now moving on to the next integral, we have

\begin{align*}
 I_2&:= \int_{\bn}\int_{B(y,m)^{\complement}}(f(x)-f(y))^2 \mathcal{K}_{s}(d(x,y))\dvg(x)\dvg(y)\\
    &\leq C \int_{\bn}\int_{B(y,m)^{\complement}}\frac{(f(x)-f(y))^2}{d(x,y)^{1+s}e^{(N-1)d(x,y)}}\dvg(x)\dvg(y)\\
    &\leq C \int_{\bn}(f(y))^2 \int_{B(y,m)^{\complement}}\frac{1}{d(x,y)^{1+s}e^{(N-1)d(x,y)}}\dvg(x)\dvg(y)\\
    &\hspace{1cm}+\int_{\bn}(f(x))^2\int_{B(x,m)^{\complement}}\frac{1}{d(x,y)^{1+s}e^{(N-1)d(x,y)}}\dvg(y)\dvg(x)\\
    &\leq C \|f\|_{L^2(\bn)}^2\int_{B(0,m)^\complement}\frac{1}{d(z,0)^{1+s}e^{(N-1)d(z,0)}}\dvg(z)\\
    &= C \|f\|_{L^2(\bn)}^2\int_{\mathbb{S}^{N-1}}\int_m^\infty \frac{(\sinh r)^{N-1}}{r^{1+s}e^{(N-1)r}}\mathrm{~d}r \mathrm{~d}\sigma.
\end{align*}
Now, we use the approximation $\sinh r \thicksim e^r\;\text{for large } r$, and  chose $m$ large enough 
\begin{align*}
    \int_{m}^{\infty}\frac{(\sinh r)^{N-1}}{r^{1+s}e^{(N-1)r}}\mathrm{~d}r   \thicksim \int_{m}^{\infty} r^{-1-s} dr = \frac{m^{-s}}{s} \quad \text{for } s > 0.
\end{align*}
Thus we get
$$I_2 \leq C \|f\|_{L^2(\bn)}^2.$$
\noindent
Finally, we estimate $I_3$ as follows
\begin{align*}
    I_3:=&\int_{\bn}\int_{B(y,m)\setminus B(y,l)}(f(x)-f(y))^2 \mathcal{K}_{s}(d(x,y))\dvg(x)\dvg(y)\\
    & \leq C\|f\|_{L^2(\bn)}^2\int_{B(0,m)\setminus B(0,l)}\mathcal{K}_{s}(d(z,0))\dvg(z).
\end{align*}
For any $z \in \bn$ such that $d(z,0)>l$, we have 
$d(z,0)>l>\tilde{l}= d(q,0)$ for some $\tilde{l}>0, q \in \bn$. Thus using Lemma~\ref{RadDecreasingLemma}, we get 
$$\mathcal{K}_{s}(d(z,0)) \leq \mathcal{K}_{s}(d(q,0)) \;\;\forall z \in \bn \text{ such that } d(z,0)>l>d(q,0).$$
Therefore, we get
\begin{align*}
    I_3 \leq \mathcal{K}_{s}(d(q,0))C\|f\|_{L^2(\bn)}^2\int_{B(0,m)\setminus B(0,l)}\dvg(z) \lesssim \|f\|_{L^2(\bn)}^2.
\end{align*}
\noindent
\noindent
Therefore, combining all the above estimates of $I_1,I_2, I_3$, and using
\eqref{firsteigen} 
\begin{align}
   &\int_{\bn}\int_{\bn}(f(x)-f(y))^2 \mathcal{K}_{s}(d(x,y))\dvg(x)\dvg(y)\notag\\
   &\le C\Bigl(\|\nabla_{\bn} f\|_{L^2(\bn)}^2+\|f\|_{L^2(\bn)}^2\Bigr)\le C\Bigl(1+\frac{4}{(N-1)^2}\Bigr)\|\nabla_{\bn} f\|_{L^2(\bn)}^2
   \label{eq:CcEstimate}
\end{align}
for every $f\in C_c^\infty(\bn)$, with $C$ independent of $f$.

For $f\in\Hh$, take $f_n\in C_c^\infty(\bn)$ with $f_n\to f$ in $\Hh$. By
\eqref{firsteigen}, $f_n\to f$ in $L^2(\bn)$, so after passing to a subsequence
$f_n\to f$ a.e. on $\bn$. Then
$(f_n(x)-f_n(y))^2\mathcal K_s(d(x,y))\to(f(x)-f(y))^2\mathcal K_s(d(x,y))$ for a.e. $(x,y)$. These functions are
nonnegative, so by Fatou's lemma, \eqref{eq:CcEstimate} applied to each $f_n$, and
$\|\nabla_{\bn}f_n\|_{L^2(\bn)}\to\|\nabla_{\bn}f\|_{L^2(\bn)}$, the theorem follows.
\end{proof}

\begin{remark}\label{rem:hatHs}
    Moreover, if we define the fractional Laplacian using the Fourier transform \eqref{FourierTranformDefn}, the above embedding naturally follows from the Plancherel formula \eqref{plancheralFormulaEqn} in a straightforward way as follows.\\
    The fractional Sobolev space can be defined as 
    \begin{equation*}
    \widehat{H}^s(\bn) = \left\{ f \in L^2(\bn) \mid \|f\|_{L^2(\bn)} + \|(-\Delta_{\bn})^{s / 2} f\|_{L^2(\bn)} < \infty \right\},
\end{equation*}
where $(-\Delta_{\bn})^{s / 2}$ is understood in the sense of \eqref{FourierTranformDefn}.
Then for $N\geq 3$, $f\in H^1(\bn)$, we compute
\begin{align*}
    \int_{\bn} |(-\Delta_{\bn})^{s/2} f|^2 \dvg(x) &= \int_{\mathbb{R} \times \mathbb{S}^{N-1}} |(\widehat{-\Delta_{\bn})^{s/2} f}(\beta,\theta)|^2 \frac{\mathrm{d} \beta\mathrm{d} \theta}{|c(\beta)|^2} \\
    &= \int_{\mathbb{R} \times \mathbb{S}^{N-1}} \left(\beta^2 + \frac{(N-1)^2}{4}\right)^s |\hat{f}(\beta, \theta)|^2 \frac{\mathrm{d} \beta\mathrm{d} \theta}{|c(\beta)|^2} \\
    &\leq \int_{\mathbb{R} \times \mathbb{S}^{N-1}} \left(\beta^2 + \frac{(N-1)^2}{4}\right) |\hat{f}(\beta, \theta)|^2 \frac{\mathrm{d} \beta\mathrm{d} \theta}{|c(\beta)|^2} \\
    &= \int_{\mathbb{R} \times \mathbb{S}^{N-1}} \hat{f}(\beta, \theta) \widehat{(-\Delta_{\bn}) \bar{f}}(\beta, \theta) \frac{\mathrm{d} \beta\mathrm{d} \theta}{|c(\beta)|^2} \\
    &= \int_{\bn} f(x) (-\Delta_{\bn} \bar{f}(x)) \dvg(x)\\
    &= \int_{\bn} |\nabla_{\bn} f|^2 \dvg(x).
\end{align*}
\end{remark}

\medskip
We have two descriptions of $(-\lap)^s$, the spectral one
\eqref{FourierTranformDefn} and the pointwise one \eqref{PVdefintion}, and
correspondingly two fractional Sobolev spaces, $H^s(\bn)$ and
$\widehat{H}^s(\bn)$. We now show that the two
spaces coincide.
\begin{proposition}[Equivalence of the spectral and pointwise definitions]
\label{prop:PVspectralEquiv}
Let $N\ge2$ and $s\in(0,1)$, and let $\Ks$ be the kernel mentioned in
\eqref{PVdefintion}.
\begin{enumerate}
\item[\textup{(i)}] For every $u\in C_c^\infty(\bn)$, the principal value in
\eqref{PVdefintion} converges at every $x\in\bn$, and
\begin{equation*}
   \mathrm{P.V.}\int_{\bn}\bigl(u(x)-u(\xi)\bigr)\Ks(d(x,\xi))\,\dvg(\xi)
   \;=\;(-\lap)^su(x)
   \quad\text{for almost every }x\in\bn,
\end{equation*}
the right hand side being understood in the sense of
\eqref{FourierTranformDefn}.
 
\item[\textup{(ii)}] For every $f\in L^2(\bn)$,
\begin{equation}\label{eq:formidentity}
   [f]_s^2
   =\int_{\R\times\Sph}
     \Bigl(\beta^2+\frac{(N-1)^2}{4}\Bigr)^{s}\,
     \bigl|\hat f(\beta,\theta)\bigr|^2\,
     \frac{\mathrm{d}\beta\,\mathrm{d}\theta}{|c(\beta)|^2}.
\end{equation} 
In particular
\eqref{eq:formidentity} holds for every $u\in\Hh$, where by
Theorem~\ref{1tosTheorem} both sides are finite, and there
\begin{equation*}
   [u]_s^2=\bigl\|(-\lap)^{s/2}u\bigr\|_{L^2(\bn)}^2 .
\end{equation*}
 
\item[\textup{(iii)}] Consequently, $H^s(\bn)=\widehat{H}^s(\bn)$, and for every
$f$ in this space 
\begin{equation*}
   \|f\|_{H^s(\bn)}^2
   =\|f\|_{L^2(\bn)}^2+\bigl\|(-\lap)^{s/2}f\bigr\|_{L^2(\bn)}^2 .
\end{equation*}
\end{enumerate}
\end{proposition}

\begin{proof}
We prove \textup{(i)} first, and deduce \textup{(ii)} and \textup{(iii)} afterwards.
 
\smallskip
Let $u\in C_c^\infty(\bn)$ and $x\in\bn$. We apply
\cite[Theorem~1.3]{fractionalplaplacianhyperbolicspaces} with $p=2$, for which $\tfrac{2}{2-s}<2$. It gives
\begin{equation}\label{eq:KKLsemigroup}
   \mathrm{P.V.}\int_{\bn}\bigl(u(x)-u(\xi)\bigr)\Ks(d(x,\xi))\,\dvg(\xi)
   = \frac{1}{|\Gamma(-s)|}\int_0^\infty
     e^{t\lap}\bigl[u(x)-u(\cdot)\bigr](x)\,\frac{\mathrm{d}t}{t^{1+s}} .
\end{equation}
Here the left hand side converges at every $x\in\bn$, since
$u\in C_c^\infty(\bn)\subset C_b^2(\bn)$
\cite[\S1, after Proposition~1.2]{fractionalplaplacianhyperbolicspaces}.
By \eqref{eq:heatkernel}, then \eqref{eq:stochcomplete} and the boundedness of $u$, we obtain
\begin{equation*}
   e^{t\lap}\bigl[u(x)-u(\cdot)\bigr](x)
   =u(x)
    -\int_{\bn}p_{t,N}\bigl(d(x,\xi)\bigr)u(\xi)\,\dvg(\xi).
\end{equation*}

Hence
\begin{equation}\label{eq:semigroupform}
   \mathrm{P.V.}\int_{\bn}\bigl(u(x)-u(\xi)\bigr)\Ks(d(x,\xi))\,\dvg(\xi)
   =\frac{1}{|\Gamma(-s)|}\int_0^\infty
     \bigl(u-e^{t\lap}u\bigr)(x)\,\frac{\mathrm{d}t}{t^{1+s}} .
\end{equation}

\smallskip
For $\zeta>0$, we have the following identity, which easily follows from
integration by parts the integral for $\Gamma(1-s)$,
\begin{equation}\label{eq:elementary}
   \frac{1}{|\Gamma(-s)|}\int_0^\infty\bigl(1-e^{-t\zeta}\bigr)\,
   \frac{\mathrm{d}t}{t^{1+s}}=\zeta^{\,s} .
\end{equation}

Next, we note the following two bounds. The first is that, for $x\in\bn$ and
$t>0$,
$|u(x)-e^{t\lap}u(x)|\le|u(x)|+|e^{t\lap}u(x)|\le2\|u\|_{L^\infty(\bn)}$ which follows from the positivity of the heat kernel, \eqref{eq:heatkernel} and \eqref{eq:stochcomplete}.
 
\smallskip
For the second bound, note that, for
$r>0$ the function $(r,\xi)\mapsto p_{r,N}(d(x,\xi))$ is smooth and satisfies
$\partial_r\,p_{r,N}(d(x,\xi))=\lap\,p_{r,N}(d(x,\xi))$, with $\lap$ acting in
the variable $\xi$. Moreover, $e^{\eps\lap}u(x)\to u(x)$ as $\eps\to0^+$ (\cite[\S2.2]{GrigoryanSurvey}). Thus,
\begin{align*}
   \frac{\mathrm{d}}{\mathrm{d}r}\,e^{r\lap}u(x)
   &=\int_{\bn}\partial_r p_{r,N}\bigl(d(x,\xi)\bigr)u(\xi)\dvg(\xi)
    =\int_{\bn}\lap p_{r,N}\bigl(d(x,\xi)\bigr)u(\xi)\dvg(\xi)\\
   &=\int_{\bn}p_{r,N}\bigl(d(x,\xi)\bigr)\lap u(\xi)\dvg(\xi)
    =e^{r\lap}\lap u(x),
\end{align*}
since
$\lap u\in C_c^\infty(\bn)$. By the fundamental theorem of calculus, we obtain
\begin{equation*}
   u(x)-e^{t\lap}u(x)=-\int_0^te^{r\lap}\lap u(x)\,\mathrm{d}r .
\end{equation*}
Combining this with the first bound, we arrive at
\begin{equation}\label{eq:ptwiseincrement}
   \bigl|u(x)-e^{t\lap}u(x)\bigr|
   \le\min\bigl\{2\|u\|_{L^\infty(\bn)},\,t\,\|\lap u\|_{L^\infty(\bn)}\bigr\},
   \quad x\in\bn,\ t>0 .
\end{equation}
Hence for every $x\in\bn$ the function
\begin{equation}\label{eq:vdef}
   v(x):=\frac{1}{|\Gamma(-s)|}\int_0^\infty
      \bigl(u(x)-e^{t\lap}u(x)\bigr)\,\frac{\mathrm{d}t}{t^{1+s}},
   \quad x\in\bn,
\end{equation}
satisfies
\begin{equation*}
   |v(x)|\le\frac{1}{|\Gamma(-s)|}\Bigl(\|\lap u\|_{L^\infty(\bn)}
     \int_0^1\frac{\mathrm{d}t}{t^{s}}
     +2\|u\|_{L^\infty(\bn)}\int_1^\infty\frac{\mathrm{d}t}{t^{1+s}}\Bigr)
   <\infty
\end{equation*}
because $0<s<1$.
 
\smallskip
In view of \eqref{eq:semigroupform}, it remains to show that $v=(-\lap)^su$.
Since $\zeta^s\le1+\zeta$ for $\zeta>0$, \eqref{eq: laplacian fourier transfom}
and \eqref{plancheralFormulaEqn} give
\begin{equation} \int_{\R\times\Sph}\Bigl(\beta^2+\frac{(N-1)^2}{4}\Bigr)^{2s}
     |\hat u|^2\,\frac{\mathrm{d}\beta\,\mathrm{d}\theta}{|c(\beta)|^2}
   \le2\|u\|_{L^2(\bn)}^2+2\|\lap u\|_{L^2(\bn)}^2<\infty , \label{another bound}
\end{equation}
so $(-\lap)^su\in L^2(\bn)$ by \eqref{FourierTranformDefn}. For real valued
$\varphi\in C_c^\infty(\bn)$,
\begin{align*}
   \int_{\bn}v\,\varphi\,\dvg
   &=\frac{1}{|\Gamma(-s)|}\int_0^\infty\!\!\int_{\bn}
     \bigl(u-e^{t\lap}u\bigr)\varphi\,\dvg\,\frac{\mathrm{d}t}{t^{1+s}}\\
   &=\frac{1}{|\Gamma(-s)|}\int_0^\infty\!\!\int_{\R\times\Sph}
     \Bigl(1-e^{-t\left(\beta^2+\frac{(N-1)^2}{4}\right)}\Bigr)
     \hat u\,\overline{\hat\varphi}\,
     \frac{\mathrm{d}\beta\,\mathrm{d}\theta}{|c(\beta)|^2}\,
     \frac{\mathrm{d}t}{t^{1+s}}\\
   &=\int_{\R\times\Sph}\Bigl(\beta^2+\frac{(N-1)^2}{4}\Bigr)^{s}
     \hat u\,\overline{\hat\varphi}\,
     \frac{\mathrm{d}\beta\,\mathrm{d}\theta}{|c(\beta)|^2}
   =\int_{\bn}(-\lap)^su\,\varphi\,\dvg .
\end{align*}
The first equality is Fubini's theorem, admissible by
\eqref{eq:ptwiseincrement} since $\varphi\in L^1(\bn)$. The second and 
fourth are \eqref{eq:parseval} applied
with \eqref{eq:heatdef} and \eqref{FourierTranformDefn}, respectively.  Here
$u-e^{t\lap}u\in L^2(\bn)$ because the multiplier in \eqref{eq:heatdef} lies in
$(0,1]$. The third is again an application of the Fubini's theorem due to the bound \eqref{another bound}, and \eqref{eq:elementary}.

Hence the local integrability of the function $v-(-\lap)^su$ along with
\eqref{eq:vdef} and \eqref{eq:semigroupform} proves \textup{(i)}.

\smallskip
We turn to \textup{(ii)}. The only property of $\Ks$ needed is
\begin{equation}\label{eq:kernelsubord}
   \Ks(\rho)=\frac{1}{|\Gamma(-s)|}\int_0^\infty p_{t,N}(\rho)\,
     \frac{\mathrm{d}t}{t^{1+s}}\qquad\text{for almost every }\rho>0 ,
\end{equation}
which we deduce from \eqref{eq:KKLsemigroup}. Fix $x\in\bn$, let
$\eta\in C_c^\infty\bigl((0,\infty)\bigr)$ with $\eta\ge0$, and put
$\varphi(\xi)=\eta(d(x,\xi))$, so that $\varphi\in C_c^\infty(\bn)$,
$\varphi\ge0$, and $\varphi$ vanishes near $x$. Applying
\eqref{eq:KKLsemigroup} to $u=-\varphi$, then \eqref{eq:heatkernel} and
Tonelli's theorem, we get
\begin{align*}
   \int_{\bn}\varphi(\xi)\,\Ks\bigl(d(x,\xi)\bigr)\,\dvg(\xi)
   &=\frac{1}{|\Gamma(-s)|}\int_0^\infty e^{t\lap}\varphi(x)\,
     \frac{\mathrm{d}t}{t^{1+s}}\\
   &=\frac{1}{|\Gamma(-s)|}\int_0^\infty\!\!\int_{\bn}
     p_{t,N}\bigl(d(x,\xi)\bigr)\varphi(\xi)\dvg(\xi)\,
     \frac{\mathrm{d}t}{t^{1+s}}\\
    &=\int_{\bn}\varphi(\xi)\,G\bigl(d(x,\xi)\bigr)\dvg(\xi),
\end{align*}
where $G(r):=\frac{1}{|\Gamma(-s)|}\int_0^\infty p_{t,N}(r)\,
t^{-1-s}\,\mathrm{d}t\in[0,\infty]$. In geodesic polar coordinates, for every $\eta\in C_c^\infty\bigl((0,\infty)\bigr),\ \eta\ge0$, we have
\begin{equation*}
   \int_0^\infty\eta(r)\,\Ks(r)\sinh^{N-1}r\,\mathrm{d}r
   =\int_0^\infty\eta(r)\,G(r)\sinh^{N-1}r\,\mathrm{d}r.
\end{equation*}
The left hand side is finite, since $\Ks$ is decreasing by
Lemma~\ref{RadDecreasingLemma}.
This gives $\Ks=G$ almost everywhere on $(0,\infty)$, which is
\eqref{eq:kernelsubord}.

\smallskip
Furthermore, fix $t>0$. For $f\in L^2(\bn)$, the function $e^{t\lap}f$ is
defined through its Fourier transform by \eqref{eq:heatdef}. For $x\in\bn$, the Cauchy-Schwarz inequality followed by \eqref{eq:stochcomplete}, gives
\begin{equation*}
\begin{aligned}
   \Bigl(\int_{\bn}p_{t,N}\bigl(d(x,\xi)\bigr)|f(\xi)|\,\dvg(\xi)\Bigr)^2
   &\leq \int_{\bn}p_{t,N}\bigl(d(x,\xi)\bigr)|f(\xi)|^2\,\dvg(\xi) .
\end{aligned}
\end{equation*}
Integrating this inequality in $x$ over $\bn$, and using \eqref{eq:stochcomplete}, we obtain
\begin{equation}\label{eq:schur}
\begin{aligned}
   &\int_{\bn}\Bigl(\int_{\bn}p_{t,N}\bigl(d(x,\xi)\bigr)|f(\xi)|\,
     \dvg(\xi)\Bigr)^2\dvg(x)\\
   &\le\int_{\bn}\Bigl(\int_{\bn}p_{t,N}\bigl(d(x,\xi)\bigr)|f(\xi)|^2\,
     \dvg(\xi)\Bigr)\dvg(x)\\
   &=\int_{\bn}|f(\xi)|^2\Bigl(\int_{\bn}p_{t,N}\bigl(d(x,\xi)\bigr)\,
     \dvg(x)\Bigr)\dvg(\xi)
   =\|f\|_{L^2(\bn)}^2, 
\end{aligned}
\end{equation}
i.e.,
\begin{equation*}
   \Bigl\|\int_{\bn}p_{t,N}\bigl(d(\cdot,\xi)\bigr)f(\xi)\,\dvg(\xi)
   \Bigr\|_{L^2(\bn)}\le\|f\|_{L^2(\bn)} .
\end{equation*}
Combining this bound with the contraction property of $e^{t\lap}$ on $L^2(\bn)$
and the density of $C_c^\infty(\bn)$ in $L^2(\bn)$, it is not difficult to show
that \eqref{eq:heatkernel} holds for every $f\in L^2(\bn)$ and almost every
$x\in\bn$. In particular, $e^{t\lap}f$ is real valued when $f$ is.

\smallskip
Now let $f\in L^2(\bn)$ be real valued. Then
\begin{small}
\begin{equation*}
   \frac12\int_{\bn}\!\int_{\bn}\bigl(f(x)-f(\xi)\bigr)^2
     p_{t,N}\bigl(d(x,\xi)\bigr)\dvg(\xi)\dvg(x)
   =\|f\|_{L^2(\bn)}^2-\int_{\bn}f\,e^{t\lap}f\dvg .
\end{equation*}
\end{small}
This follows by expanding
$\bigl(f(x)-f(\xi)\bigr)^2=f(x)^2-2f(x)f(\xi)+f(\xi)^2$. By the symmetry of the
heat kernel and \eqref{eq:stochcomplete}, each of the two square terms
contributes $\|f\|_{L^2(\bn)}^2$. By Fubini's theorem and \eqref{eq:heatkernel},
the cross term contributes $-2\int_{\bn}f\,e^{t\lap}f\,\dvg$. Indeed, the cross term converges absolutely by the Cauchy-Schwarz
and \eqref{eq:schur} 
\begin{align*}
   &\int_{\bn}\!\int_{\bn}|f(x)|\,|f(\xi)|\,p_{t,N}\bigl(d(x,\xi)\bigr)
     \dvg(\xi)\dvg(x)\\
     &\le\|f\|_{L^2(\bn)}\Bigl\|\int_{\bn}p_{t,N}\bigl(d(\cdot,\xi)\bigr)
     |f(\xi)|\dvg(\xi)\Bigr\|_{L^2(\bn)}\le\|f\|_{L^2(\bn)}^2.
\end{align*}
By \eqref{eq:parseval} for the real valued pair $f,\,e^{t\lap}f$, with
$\widehat{e^{t\lap}f}=e^{-t\left(\beta^2+\frac{(N-1)^2}{4}\right)}\hat f$ by
\eqref{eq:heatdef}, and by \eqref{plancheralFormulaEqn},
\begin{small}
\begin{equation*}
   \int_{\bn}f\,e^{t\lap}f\,\dvg
   =\int_{\R\times\Sph}e^{-t\left(\beta^2+\frac{(N-1)^2}{4}\right)}
     \bigl|\hat f\bigr|^2\,\frac{\mathrm{d}\beta\,\mathrm{d}\theta}{|c(\beta)|^2},
   \,\,
   \|f\|_{L^2(\bn)}^2
   =\int_{\R\times\Sph}\bigl|\hat f\bigr|^2\,
     \frac{\mathrm{d}\beta\,\mathrm{d}\theta}{|c(\beta)|^2}.
\end{equation*}
\end{small}
Subtracting these two finite quantities,
\begin{align}\label{eq:heatenergy}
  &\frac12\int_{\bn}\int_{\bn}\bigl(f(x)-f(\xi)\bigr)^2
     p_{t,N}\bigl(d(x,\xi)\bigr)\dvg(\xi)\dvg(x) \notag\\
   &  \hspace{2cm}= \int_{\R\times\Sph}\Bigl(1-e^{-t\left(\beta^2+\frac{(N-1)^2}{4}\right)}
     \Bigr)\bigl|\hat f\bigr|^2
     \frac{\mathrm{d}\beta\mathrm{d}\theta}{|c(\beta)|^2} .
\end{align}
 
\smallskip
Further, using \eqref{eq:kernelsubord}, 
\begin{equation*}
   [f]_s^2=\frac{1}{|\Gamma(-s)|}\int_0^\infty
     \biggl(\frac12\int_{\bn}\!\int_{\bn}\bigl(f(x)-f(\xi)\bigr)^2
       p_{t,N}\bigl(d(x,\xi)\bigr)\,\dvg(\xi)\,\dvg(x)\biggr)
     \frac{\mathrm{d}t}{t^{1+s}} .
\end{equation*}
Inserting \eqref{eq:heatenergy} and applying Tonelli 
\begin{equation*}
   [f]_s^2=\int_{\R\times\Sph}\biggl(\frac{1}{|\Gamma(-s)|}\int_0^\infty
     \Bigl(1-e^{-t\left(\beta^2+\frac{(N-1)^2}{4}\right)}\Bigr)
     \frac{\mathrm{d}t}{t^{1+s}}\biggr)\bigl|\hat f\bigr|^2\,
     \frac{\mathrm{d}\beta\,\mathrm{d}\theta}{|c(\beta)|^2} ,
\end{equation*}
and the inner integral is $\bigl(\beta^2+\frac{(N-1)^2}{4}\bigr)^{s}$ by
\eqref{eq:elementary}. This is
\eqref{eq:formidentity}.
 
\smallskip
Finally, we prove \textup{(iii)}. By \textup{(ii)}, $[f]_s^2$ equals the right hand side
of \eqref{eq:formidentity} in $[0,\infty]$, so the space $H^s(\bn)$  defined in \eqref{eq:spectral fractional} coincides with the space $\widehat{H}^s(\bn)$ of
Remark~\ref{rem:hatHs}. For $f$ in this space, \eqref{FourierTranformDefn} and
\eqref{plancheralFormulaEqn} give $[f]_s=\|(-\lap)^{s/2}f\|_{L^2(\bn)}$.
\end{proof}
 \smallskip
As a first consequence of Proposition~\ref{prop:PVspectralEquiv}, the seminorm $[u]_s$
can be compared with the $L^2$ norm and the Dirichlet energy of $u$ through the
Helgason--Fourier transform. 
\begin{lemma}\label{lem:interp}
For every $u\in H^1(\bn)$,
\begin{equation}\label{eq:gradFourier}
   \int_{\R\times\Sph}\Bigl(\beta^2+\frac{(N-1)^2}{4}\Bigr)|\hat u|^2
   \frac{\mathrm{d}\beta\,\mathrm{d}\theta}{|c(\beta)|^2}
   =\int_{\bn}|\nabla_{\bn} u|^2\dvg ,
\end{equation}
and
\begin{small}
\begin{equation}\label{eq:interp}
   \Bigl(\frac{(N-1)^2}{4}\Bigr)^{s}\|u\|_{L^2(\bn)}^2
   \le[u]_s^2
   \le\|u\|_{L^2(\bn)}^{2(1-s)}\,\|\nabla_{\bn} u\|_{L^2(\bn)}^{2s}
   \le\Bigl(\frac{(N-1)^2}{4}\Bigr)^{s-1}\|\nabla_{\bn} u\|_{L^2(\bn)}^2 .
\end{equation}
\end{small}
\end{lemma}

\begin{proof}
For
$u\in C_c^\infty(\bn)$, by the divergence theorem, \eqref{eq:parseval} and
\eqref{eq: laplacian fourier transfom},
\begin{small}
\begin{equation*}
   \int_{\bn}|\nabla_{\bn} u|^2\dvg=\int_{\bn}(-\lap u)u\dvg
   =\int_{\R\times\Sph}\left(\beta^2+\frac{(N-1)^2}{4}\right)|\hat u|^2\,|c(\beta)|^{-2}\mathrm{d}\beta\mathrm{d}\theta.
\end{equation*}
\end{small}
For $u\in H^1(\bn)$, \eqref{eq:gradFourier} follows by density of $C_c^\infty(\bn)$, and \eqref{plancheralFormulaEqn}.
The first inequality in \eqref{eq:interp} follows from
$\bigl(\beta^2+\frac{(N-1)^2}{4}\bigr)^{s}\ge\bigl(\frac{(N-1)^2}{4}\bigr)^{s}$ and
\eqref{plancheralFormulaEqn}. For the second, write
\begin{equation*}
   \Bigl(\beta^2+\frac{(N-1)^2}{4}\Bigr)^s|\hat u|^2
   =\bigl(|\hat u|^2\bigr)^{1-s}\Bigl(\Bigl(\beta^2+\frac{(N-1)^2}{4}\Bigr)|\hat u|^2\Bigr)^s
\end{equation*}
and apply H\"older's inequality with exponents $\frac1{1-s}$ and $\frac1s$, together with
\eqref{plancheralFormulaEqn} and \eqref{eq:gradFourier}. The third follows from
\eqref{firsteigen}.
\end{proof}

\medskip
\noindent We now compare the best Sobolev constants.  Having established the continuous embedding, the associated best constant
$S_{\lambda,s}$ of \eqref{eq:SK} is well defined and positive. 
Recall the sharp Sobolev inequality on $\RN$ \textup{(}\cite{AubinT, TalentiG}\textup{)}:
for every $w\in\mathcal{D}^{1,2}(\RN)$,
\begin{equation*}
S\Bigl(\int_{\RN}|w|^{\cstar}\,dx\Bigr)^{2/\cstar}\le\int_{\RN}|\nabla w|^2\,dx,
\end{equation*}
where $\mathcal{D}^{1,2}(\RN)$ denotes the completion of $C_c^\infty(\RN)$ under
$\|\nabla\,\cdot\,\|_{L^2(\RN)}$, and
\begin{equation}\label{eq:bestSobolevConstantEuclidean}
   S=\inf_{\substack{w\in\mathcal{D}^{1,2}(\RN)\\ w\ne0}}
     \frac{\int_{\RN}|\nabla w|^2\,dx}{\bigl(\int_{\RN}|w|^{\cstar}\,dx\bigr)^{2/\cstar}}.
\end{equation}
The above infimum is attained by the Aubin-Talenti family
$c\,U\bigl((x-x_0)/\mu\bigr)$, $c\in\R\setminus\{0\}$, $\mu>0$, $x_0\in\RN$, where the
normalized bubble
\begin{equation*}
   U(z)=\bigl[N(N-2)\bigr]^{\frac{N-2}{4}}(1+|z|^2)^{-\frac{N-2}{2}}
   \quad\text{satisfies}\quad
   \int_{\RN}|\nabla U|^2\,dx=\int_{\RN}U^{\cstar}\,dx=S^{N/2}.
\end{equation*}

\medskip

Step~1a of the proof of Theorem~\ref{thm:sobolevConstants} uses that the Aubin--Talenti bubble $U$ has finite
Gagliardo seminorm
$[U]_{s,\RN}^2:=\iint_{\RN\times\RN}\frac{(U(x)-U(y))^2}{|x-y|^{N+2s}}\,dx\,dy$.
The estimates in \cite[Remark~3.2]{biagi} give this for $N\ge5$, since they use
$U\in L^2(\RN)$. In Lemma~\ref{lem:Ufinite} we show that $[U]_{s,\RN}$ is also
finite for $N=4$, and for $N=3$ with $s>\frac12$. It is infinite for
$N=3$ with $s\le\frac12$, and in this case, we close Step~1a by interpolation.

\begin{lemma}\label{lem:Ufinite}
Let $N\ge3$ and $s\in(0,1)$. Then $[U]_{s,\RN}<\infty$ if and only if $N+2s>4$,
that is, if and only if $N\ge4$, or $N=3$ and $s>\frac12$.
\end{lemma}

\begin{proof}
Write
\begin{equation}\label{eq:Uxy}
   [U]_{s,\RN}^2=\int_{\RN}\int_{\RN}\frac{|U(x+y)-U(x)|^2}{|y|^{N+2s}}\,dx\,dy .
\end{equation}
Throughout the proof, $C$ denotes positive constants depending only on $N$ and $s$, whose value may change from line to line. As in \cite[Remark~3.2]{biagi},
\begin{equation}\label{eq:Ubiagi}
   |U(z)|\le C\min\{1,|z|^{2-N}\}\le C(1+|z|)^{2-N},\qquad
   |\nabla U(z)|\le C\min\{|z|,|z|^{1-N}\}.
\end{equation}

\smallskip
\noindent\textit{Sufficiency.} Let $N+2s>4$. We split $\RN\times\RN$ into
$\RN\times B_1$, where $B_1=\{|y|<1\}$, and the two sets
\begin{equation*}
   D_1:=\Bigl\{(x,y):|y|\ge\max\Bigl\{1,\frac{1+|x|}{2}\Bigr\}\Bigr\},
   \qquad
   D_2:=\Bigl\{(x,y):1\le|y|<\frac{1+|x|}{2}\Bigr\}.
\end{equation*}

\smallskip
On the region $\RN\times B_1$, using the mean value theorem (MVT)
\begin{equation*}
   \int_{B_1}\int_{\RN}\frac{|U(x+y)-U(x)|^2}{|y|^{N+2s}}\,dx\,dy
   \le\|\nabla U\|_{L^2(\RN)}^2\int_{B_1}\frac{dy}{|y|^{N+2s-2}}
   =\frac{|\mathbb S^{N-1}|}{2-2s}\,\|\nabla U\|_{L^2(\RN)}^2 .
\end{equation*}
By \eqref{eq:Ubiagi},
$\|\nabla U\|_{L^2(\RN)}^2\le C\int_{\RN}\min\{|z|^2,|z|^{2-2N}\}\,dz<\infty$, since
$2N-2>N$ for $N\ge3$.

\smallskip
On the region $D_1$, use $|U(x+y)-U(x)|^2\le2|U(x+y)|^2+2|U(x)|^2$.

For the term with $U(x)$, fix $x$. The $y$-section of $D_1$ is $\{|y|\ge R_x\}$ with
$R_x=\max\{1,\frac{1+|x|}{2}\}\ge\frac{1+|x|}{2}$, so
\begin{equation*}
   \int_{\{|y|\ge R_x\}}\frac{dy}{|y|^{N+2s}}
   =\frac{|\mathbb S^{N-1}|}{2s}R_x^{-2s}
   \le  C (1+|x|)^{-2s}.
\end{equation*}
By Tonelli's theorem and \eqref{eq:Ubiagi}, this term contributes at most
\begin{small}
\begin{equation}\label{eq:D1integral}
   C\hspace{-0.2em}\int_{\RN}|U(x)|^2(1+|x|)^{-2s}dx
   \le \hspace{-0.2em}C\int_{\RN}(1+|x|)^{4-2N-2s}dx
   =C|\mathbb S^{N-1}|\int_0^\infty(1+\rho)^{4-2N-2s}\rho^{N-1}d\rho .
\end{equation}
\end{small}
The last integral is finite. Since $4-2N-2s<0$, on $[0,1]$ the integrand is at most $1$.
For $\rho\ge1$, the integrand is at most
$\rho^{3-N-2s}$, and
\begin{equation*}
   \int_1^\infty\rho^{3-N-2s}\,d\rho=\frac{1}{N+2s-4}<\infty ,
\end{equation*}
because $3-N-2s<-1$, which is exactly the assumption $N+2s>4$.

For the term with $U(x+y)$, put $z=x+y$. On $D_1$ we have $|z|<3|y|$, and with $|y|\ge1$ this gives $1+|z|<4|y|$. So the image of $D_1$ under
$(x,y)\mapsto(z,y)$ lies in $\{|y|\ge\frac{1+|z|}{4}\}$, and as above this term contributes
at most $C\int_{\RN}|U(z)|^2(1+|z|)^{-2s}dz$, which is finite by \eqref{eq:D1integral}.

\smallskip
On $D_2$ we have $1\le|y|<\frac{1+|x|}{2}$, hence $|x|>1$.

Let first $|x|\ge3$. Then $|y|<\frac{2|x|}{3}$ and
$|x+ty|\ge|x|-|y|\ge\frac{|x|}{3}$ for $t\in[0,1]$. By MVT and
\eqref{eq:Ubiagi}, $|U(x+y)-U(x)|^2\le C|y|^2|x|^{2-2N}$. In polar coordinates, since
$2-2s>0$,
\begin{equation*}
   \int_{\{1\le|y|\le|x|\}}\frac{dy}{|y|^{N+2s-2}}
   =|\mathbb S^{N-1}|\int_1^{|x|}\rho^{1-2s}d\rho
   \le C |x|^{2-2s}.
\end{equation*}
Hence the contribution of $D_2\cap\{|x|\ge3\}$ is at most
\begin{equation*}
   C\int_{\{|x|\ge3\}}|x|^{2-2N}\int_{\{1\le|y|\le|x|\}}\frac{dy}{|y|^{N+2s-2}}\,dx
   \leq C\int_3^\infty\rho^{3-N-2s}d\rho<\infty ,
\end{equation*}
again because $N+2s>4$.

Let now $1<|x|<3$. Then $1\le|y|<2$, and by \eqref{eq:Ubiagi}
$|U(x+y)-U(x)|^2|y|^{-N-2s}\le4C^2$. So the contribution of $D_2\cap\{|x|<3\}$ is at
most $4C^2\,|\{1<|x|<3\}|\,|\{1\le|y|<2\}|<\infty$.

Adding the three regions, $[U]_{s,\RN}<\infty$.

\smallskip
\noindent\textit{Necessity.} Let $N+2s\le4$, that is, $N=3$ and $s\le\frac12$. Consider
\begin{equation*}
   E:=\bigl\{(x,y):|x|\ge1,\ |y|\ge4|x|\bigr\}.
\end{equation*}
Let $(x,y)\in E$. Then $|x+y|\ge3|x|$, so $1+|x+y|^2>9|x|^2$, while
$1+|x|^2\le2|x|^2$, then
\begin{equation*}
   0<\frac{U(x+y)}{U(x)}
   =\Bigl(\frac{1+|x|^2}{1+|x+y|^2}\Bigr)^{\frac{N-2}{2}}
   <\Bigl(\frac{2}{9}\Bigr)^{\frac{N-2}{2}}
   \le\Bigl(\frac29\Bigr)^{\frac12}<\frac12 .
\end{equation*}
Hence
$|U(x+y)-U(x)|\ge\frac12U(x)$. Here $U(z)=\alpha_N(1+|z|^2)^{-\frac{N-2}{2}}$ with $\alpha_N=[N(N-2)]^{\frac{N-2}{4}}$. Moreover,
\begin{equation*}
   U(x)^2=\alpha_N^2\bigl(1+|x|^2\bigr)^{-(N-2)}\ge\alpha_N^2\,2^{2-N}|x|^{4-2N},
   \quad
   \int_{\{|y|\ge4|x|\}}\frac{dy}{|y|^{N+2s}}=\frac{|\mathbb S^{N-1}|}{2s}(4|x|)^{-2s}.
\end{equation*}
The integrand in \eqref{eq:Uxy} is nonnegative, so by Tonelli's theorem
\begin{equation*}
   [U]_{s,\RN}^2
   \ge\frac14\int_{\{|x|\ge1\}}U(x)^2\int_{\{|y|\ge4|x|\}}\frac{dy}{|y|^{N+2s}}\,dx
   \ge C\,|\mathbb S^{N-1}|\int_1^\infty\rho^{3-N-2s}d\rho.
\end{equation*}
The last integral is infinite, because $3-N-2s\ge-1$.
\end{proof}

\begin{proof}[Proof of Theorem \ref{thm:sobolevConstants}]
We prove the three parts in turn. In part \textup{(i)}, the upper bound
$\Smix\le S$ is obtained by testing the mixed quotient with a concentrating
Aubin--Talenti bubble, and the lower bound and the non-attainment follow from a
conformal identity for the local energy combined with the lower bound on
$[u]_s^2$ provided by Proposition~\ref{prop:PVspectralEquiv}. Parts \textup{(ii)}
and \textup{(iii)} then follow from \textup{(i)}.

\medskip
\noindent\textbf{Proof of Part \textup{(i)}.}

Let $\varphi\in C_c^\infty(\bn)$ be radial with $0\le\varphi\le1$, $\varphi\equiv1$ on
$\{|x|\le r/2\}$ and $\supp\varphi\subset K:=\{|x|\le r\}$ for some $0<r<1$. Fix
$r_1\in(r,1)$ and set $K_1:=\{|x|\le r_1\}$, so that $K\subset\subset K_1\subset\subset\bn$.
For $\varepsilon>0$ set
\begin{equation}
   w_\varepsilon(x)=\varepsilon^{-\frac{N-2}{2}}U(x/\varepsilon),\qquad
   v_\varepsilon=\varphi\,w_\varepsilon,\qquad
   u_\varepsilon=\Bigl(\tfrac{2}{1-|x|^2}\Bigr)^{-\frac{N-2}{2}}v_\varepsilon.
\label{eq:concAubinTalentiBubbles}
\end{equation}

\smallskip
\noindent\textbf{Step 1: Estimate of the nonlocal term.} Throughout,
$\Theta(x):=\bigl(\tfrac{2}{1-|x|^2}\bigr)^{-(N-2)/2}$ for the conformal factor, so that
$u_\varepsilon=\Theta\,v_\varepsilon$ and $\supp u_\varepsilon\subset K$. On the compact set $K$, $\Theta$ is smooth and satisfies
\begin{equation}\label{eq:Theta}
  C(N,r_1) \leq \Theta(x)\le\Theta_{\max},
   \qquad  |\Theta(x)-\Theta(y)|\le L_\Theta|x-y|\quad \text{for } x,y\in K_1,
\end{equation}
with $C(N,r_1)>0$, and  $L_\Theta=L_\Theta(N,r_1):=\|\nabla\Theta\|_{L^\infty(K_1)}<\infty$.\\
Since $u_\varepsilon=\Theta\,v_\varepsilon$ is smooth with $\supp u_\varepsilon\subset K\subset\subset\bn$,
therefore $u_\varepsilon\in H^1(\bn)$ by
\cite[Lemma~2.3]{MS}. In particular, $u_\varepsilon$ is an admissible test function for $\Sloc$ and $\Smix$.

Split
\begin{equation*}
   2[u_\varepsilon]_s^2
   =\underbrace{\iint_{K_1\times K_1}}_{\mathrm{(I)}}
    +\underbrace{2\iint_{K_1\times K_1^c}}_{\mathrm{(II)}}
    +\underbrace{\iint_{K_1^c\times K_1^c}}_{\mathrm{(III)\equiv 0}}
    \bigl(u_\varepsilon(x)-u_\varepsilon(y)\bigr)^2\,\Ks(\dist(x,y))\,\dV(x)\,\dV(y).
\end{equation*}

We begin by estimating $\mathrm{(II)}$. Since $u_\varepsilon$ vanishes outside $K$ 
\begin{equation*}
   \mathrm{(II)}=2\int_{K}u_\varepsilon(x)^2
\Bigl(\int_{K_1^c}\Ks(\dist(x,y))\,\dV(y)\Bigr)\dV(x).
\end{equation*}
We now bound the inner integral uniformly in $x\in\supp u_\varepsilon$. Using
$\dist(0,z)=\log\frac{1+|z|}{1-|z|}=2\operatorname{artanh}|z|$, 
\begin{equation*}
   K=\{|x|\le r\}=B(0,R),\qquad K_1=\{|x|\le r_1\}=B(0,R_1),
\end{equation*}
where $R=2\operatorname{artanh}(r),\ \ R_1=2\operatorname{artanh}(r_1)$.
Hence, for $x\in\supp u_\varepsilon\subset K$ and $y\in K_1^c$, the triangle inequality gives
\begin{equation*}
   \dist(x,y)\ge\dist(0,y)-\dist(0,x)>R_1-R
   =2\bigl(\operatorname{artanh}(r_1)-\operatorname{artanh}(r)\bigr)=:d_0>0.
\end{equation*}
Proceeding as in estimates $I_2$ and $I_3$ in the proof of Theorem~\ref{1tosTheorem},
\begin{equation*}
   \int_{K_1^c}\Ks(\dist(x,y))\,\dV(y)
   \le\int_{\{y\,:\,\dist(x,y)\ge d_0\}}\Ks(\dist(x,y))\,\dV(y)
   \le C,
\end{equation*}
where $C=C(N,s,r,r_1)$ is independent of $x$. Therefore, since $\Theta$ is bounded on
$K$ and $0\le\varphi\le1$,
\begin{equation}
   \mathrm{(II)}\le C\int_{K}v_\varepsilon(x)^2\,dx
   \le C(N,s,r,r_1)\,\|v_\varepsilon\|_{L^2(\RN)}^2. \label{eq:reigon2final}
\end{equation}
In Region~$\mathrm{(I)}$, both $x,y\in K_1$. Using \eqref{eq:Theta}, we write
\begin{align*}
   u_\varepsilon(x)-u_\varepsilon(y)
  & =\Theta(x)\bigl(v_\varepsilon(x)-v_\varepsilon(y)\bigr)
     +v_\varepsilon(y)\bigl(\Theta(x)-\Theta(y)\bigr).
     \end{align*}
     \begin{align*}
     \implies \bigl(u_\varepsilon(x)-u_\varepsilon(y)\bigr)^2
  & \le C(N,r_1)\left[\bigl(v_\varepsilon(x)-v_\varepsilon(y)\bigr)^2
     +\,v_\varepsilon(y)^2\,|x-y|^2\right].
\end{align*}
Since $\supp v_\varepsilon\subset K$, and bounding each volume factor by $\dV\le C(N,r_1)\,dx$ on $K_1$,

\smallskip
\begin{equation}
   \mathrm{(I)}\le C\!\iint_{K_1\times K_1}\!\!\bigl(v_\varepsilon(x)-v_\varepsilon(y)\bigr)^2
        \Ks(\dist(x,y))\,dx\,dy
      +C\!\iint_{K_1\times K}\!\!v_\varepsilon(y)^2|x-y|^2\,\Ks(\dist(x,y))\,dx\,dy . \label{eq:MaintermSplit}
\end{equation}
\smallskip

On $K_1$ the geodesic distance is comparable to the Euclidean one:
$C^{-1}|x-y|\le\dist(x,y)\le C|x-y|$ with $C=C(r_1)$. Indeed, in the ball model,
\begin{equation*}
   \cosh\dist(x,y)=1+t,\qquad t:=\frac{2|x-y|^2}{(1-|x|^2)(1-|y|^2)},
\end{equation*}
and from the integral representation
$\operatorname{arccosh}(1+t)=\int_0^t\frac{d\tau}{\sqrt{\tau^2+2\tau}}$, together with the
elementary inequalities $\sqrt{2\tau}\le\sqrt{\tau^2+2\tau}\le\sqrt{2\tau}\,\sqrt{1+\tfrac t2}$
valid for $\tau\in[0,t]$, we obtain
\begin{equation}\label{eq:arccosh}
   \frac{\sqrt{2t}}{\sqrt{1+t/2}}\;\le\;\operatorname{arccosh}(1+t)\;\le\;\sqrt{2t}.
\end{equation}
For $x,y\in K_1=\{|x|\le r_1\}$ we have $(1-|x|^2)(1-|y|^2)\in\bigl[(1-r_1^2)^2,\,1\bigr]$
and $|x-y|\le 2r_1$, hence
\begin{equation}\label{eq:tbounds}
   2|x-y|^2\;\le\;t\;\le\;\frac{2|x-y|^2}{(1-r_1^2)^2},
   \qquad
   1+\frac t2\;\le\;1+\frac{4r_1^2}{(1-r_1^2)^2}=\Bigl(\frac{1+r_1^2}{1-r_1^2}\Bigr)^2 .
\end{equation}
Substituting the upper bound for $t$ into the right hand inequality of \eqref{eq:arccosh}
gives
$$
   \dist(x,y)\le\sqrt{2t}\le\frac{2}{1-r_1^2}\,|x-y|,
$$
while the lower bound for $t$ and the last estimate in \eqref{eq:tbounds} give, via the
left hand inequality of \eqref{eq:arccosh},
$$
   \dist(x,y)\ge\frac{\sqrt{2t}}{\sqrt{1+t/2}}
   \ge\frac{2|x-y|}{\bigl(1+r_1^2\bigr)/(1-r_1^2)}
   =\frac{2(1-r_1^2)}{1+r_1^2}\,|x-y|.
$$
Consequently,
\begin{equation}\label{eq:distcomp}
   \frac{2(1-r_1^2)}{1+r_1^2}\,|x-y|
   \;\le\;\dist(x,y)\;\le\;
   \frac{2}{1-r_1^2}\,|x-y|,
   \qquad x,y\in K_1 .
\end{equation}

Returning to the first integral in \eqref{eq:MaintermSplit}, we split, $K_1\times K_1$ at $\dist(x,y)=l$. On the near
part $\{\dist(x,y)<l\}$, by \eqref{eq:nearkernelestimate} and \eqref{eq:distcomp},
$$
   \Ks(\dist(x,y))\le C\,\dist(x,y)^{-(N+2s)}\le C(N,s,r_1)\,|x-y|^{-(N+2s)},
$$
while on the far part $\{\dist(x,y)\ge l\}$, and both points lie in $K_1$, so
$\dist(x,y)\le\dist(x,0)+\dist(0,y)\le 2R_1$ with $R_1=2\operatorname{artanh}(r_1)$. Hence
$\dist(x,y)$ ranges in the compact interval $[l,2R_1]$ and
$\Ks(\dist(x,y))\le C$. Thus the first
integral in the RHS of \eqref{eq:MaintermSplit} satisfies
$$
   \iint_{K_1\times K_1}(v_\varepsilon(x)-v_\varepsilon(y))^2\,\Ks(\dist(x,y))\,dx\,dy
   \le C(N,s,r_1)\left[[v_\varepsilon]_{s,\RN}^2+\,\|v_\varepsilon\|_{L^2(\RN)}^2\right],
   $$
where for measurable $u\colon\RN\to\R$,
$[u]_{s,\RN}^2:=\iint_{\RN\times\RN}\frac{(u(x)-u(y))^2}{|x-y|^{N+2s}}\,dx\,dy$ denotes the
Gagliardo seminorm. 

\medskip
For the second integral in \eqref{eq:MaintermSplit}, we proceed similarly. Fix $y\in K$ and split the inner
$x$-integral at $\dist(x,y)=l$. On the near part, by \eqref{eq:distcomp},
$|x-y|^2\Ks(\dist(x,y))\le C\,|x-y|^{2-(N+2s)}$, and since
$\dist(x,y)<l$ forces $|x-y|\le C(r_1)\,l$, again by
\eqref{eq:distcomp},
$$
   \int_{\{\dist(x,y)<l\}}|x-y|^2\,\Ks(\dist(x,y))\,dx
   \le C(N,s,r_1)\int_{|x-y|\le C l}|x-y|^{2-(N+2s)}\,dx
 <\infty,
$$
since $s<1$. On the far part $\{\dist(x,y)\ge l\}$, we have
$|x-y|\le 2r_1$, so
$|x-y|^2\Ks(\dist(x,y))\le C(N,s,r_1)$, and the inner $x$-integral is bounded,
uniformly in $y$. Therefore,
\begin{equation*}
   \iint_{K_1\times K}v_\varepsilon(y)^2|x-y|^2\,\Ks(\dist(x,y))\,dx\,dy
   \le C(N,s,r_1)\,\|v_\varepsilon\|_{L^2(\RN)}^2.
\end{equation*}

Combining the two estimates gives
\begin{equation}
   \mathrm{(I)}
   \;\le\; C(N,s,r_1)\left[[v_\varepsilon]_{s,\RN}^2+\|v_\varepsilon\|_{L^2(\RN)}^2\right]. \label{eq:SimplifiedEstimateofI}
\end{equation}

\medskip
\noindent\textbf{Step 1a: the Gagliardo energy.}
We now estimate the first term in \eqref{eq:SimplifiedEstimateofI}. Recall $v_\varepsilon=\varphi\,w_\varepsilon$ with $\supp\varphi\subset\{|x|\le r\}$. From
$v_\varepsilon(x)-v_\varepsilon(y)=\varphi(x)\bigl(w_\varepsilon(x)-w_\varepsilon(y)\bigr)
+w_\varepsilon(y)\bigl(\varphi(x)-\varphi(y)\bigr)$, we get
\begin{small}
\begin{equation}\label{eq:Esplit}
   [v_\varepsilon]_{s,\RN}^2
   \le 2\underbrace{\iint_{\RN\times\RN}\frac{\varphi(x)^2\bigl(w_\varepsilon(x)-w_\varepsilon(y)\bigr)^2}{|x-y|^{N+2s}}dx\,dy}_{=:E_1}
     +2\underbrace{\iint_{\RN\times\RN}\frac{w_\varepsilon(y)^2\bigl(\varphi(x)-\varphi(y)\bigr)^2}{|x-y|^{N+2s}}dx\,dy}_{=:E_2}.
\end{equation}
\end{small}
Since $0\le\varphi\le1$, 
\begin{equation}\label{eq:E1}
   E_1\le\iint_{\RN\times\RN}\frac{(w_\varepsilon(x)-w_\varepsilon(y))^2}{|x-y|^{N+2s}}dx\,dy
   =[w_\varepsilon]_{s,\RN}^2 .
\end{equation}

To estimate $E_2$, we first bound the inner $x$-integral
$J(y):=\int_{\RN}\frac{(\varphi(x)-\varphi(y))^2}{|x-y|^{N+2s}}dx$.

\medskip

\emph{(a) Near $y$ ($|y|\le2r$).} Split the $x$-integral defining $J(y)$ at $|x-y|=2r$.
On $|x-y|\le2r$, the Lipschitz bound $(\varphi(x)-\varphi(y))^2\le\|\nabla\varphi\|_\infty^2|x-y|^2$
gives
$$
   \int_{|x-y|\le2r}\frac{(\varphi(x)-\varphi(y))^2}{|x-y|^{N+2s}}\,dx
   \le\|\nabla\varphi\|_\infty^2\int_{|x-y|\le2r}|x-y|^{2-(N+2s)}\,dx
   \le C(N,s,r).
$$
On $|x-y|>2r$, since $0\le\varphi\le1$ 
$$
   \int_{|x-y|>2r}\frac{(\varphi(x)-\varphi(y))^2}{|x-y|^{N+2s}}\,dx
   \le\int_{|x-y|>2r}|x-y|^{-(N+2s)}\,dx\le C(N,s,r).
$$
Thus $J(y)\le C(N,s,r)$ uniformly in $y$, and
\begin{equation}\label{eq:E2near}
   \int_{|y|\le2r}w_\varepsilon(y)^2J(y)\,dy
   \le C(N,s,r)\int_{|y|\le2r}w_\varepsilon(y)^2\,dy
   = C(N,s,r)\,\|w_\varepsilon\|_{L^2(\{|y|\le2r\})}^2 .
\end{equation}
Here, for a measurable set $\Omega\subset\RN$,
$\|u\|_{L^2(\Omega)}^2:=\int_{\Omega}|u|^2\,dx$.

\medskip

\emph{(b) Far $y$ ($|y|>2r$).} Here $\varphi(y)=0$, so
$(\varphi(x)-\varphi(y))^2=\varphi(x)^2$ is supported in $\{|x|\le r\}$, and for such $x$,
$|x-y|\ge|y|-r\ge|y|/2$; hence
$J(y)\le\int_{|x|\le r}|x-y|^{-(N+2s)}\,dx\le C(N,s,r)\,|y|^{-(N+2s)}$. Therefore, with
$w_\varepsilon(y)^2=\varepsilon^{-(N-2)}U(y/\varepsilon)^2$ and $y=\varepsilon Y$,
\begin{small}
\begin{align}\label{eq:E2far}
   \int_{|y|>2r}w_\varepsilon(y)^2J(y)\,dy
   \le C\!\int_{|y|>2r}\!\frac{w_\varepsilon(y)^2}{|y|^{N+2s}}\,dy
  & =C\,\varepsilon^{2-N-2s}\!\int_{|Y|>2r/\varepsilon}\!\frac{U(Y)^2}{|Y|^{N+2s}}\,dY \notag\\
   & \hspace{2cm}\le C(N,s,r)\,\varepsilon^{N-2}.
\end{align}
\end{small}
Adding \eqref{eq:E2near} and \eqref{eq:E2far},
\begin{equation}\label{eq:E2}
   E_2\le C(N,s,r)\Bigl[\|w_\varepsilon\|_{L^2(\{|y|\le2r\})}^2+\varepsilon^{N-2}\Bigr].
\end{equation}
Combining \eqref{eq:Esplit}, \eqref{eq:E1}, \eqref{eq:E2}, and using
$[w_\varepsilon]_{s,\RN}^2=\varepsilon^{2-2s}[U]_{s,\RN}^2$,
\begin{equation}
   [v_\varepsilon]_{s,\RN}^2
   \le C(N,s,r)\Bigl[\,\varepsilon^{2-2s}[U]_{s,\RN}^2
     +\|w_\varepsilon\|_{L^2(\{|y|\le2r\})}^2+\varepsilon^{N-2}\Bigr], \label{FractFinalEst}
\end{equation}

If $N+2s>4$, then $[U]_{s,\RN}<\infty$ by Lemma~\ref{lem:Ufinite}. In that case the first
term in \eqref{FractFinalEst} is of order $\varepsilon^{2-2s}$, and the second term
$\|w_\varepsilon\|_{L^2(\{|y|\le2r\})}^2$ is estimated as in Step~1b below.\\
If $N=3$ and $s\le\frac12$, then $[U]_{s,\RN}=\infty$ by Lemma~\ref{lem:Ufinite}, and
\eqref{FractFinalEst} carries no information. We then bound $[v_\varepsilon]_{s,\RN}$
directly. Since $v_\varepsilon\in C_c^\infty(\RN)\subset H^s(\RN)$,
\cite[Proposition~3.4]{DiNezzaPalatucciValdinoci2012}, H\"older's inequality with
exponents $\frac1{1-s}$ and $\frac1s$
\begin{equation*}
[v_\varepsilon]_{s,\RN}^2=c_{N,s}\int_{\RN}\bigl(|\hat v_\varepsilon|^2\bigr)^{1-s}
   \bigl(|\xi|^2|\hat v_\varepsilon|^2\bigr)^{s}d\xi
   \le c_{N,s}\,\|v_\varepsilon\|_{L^2(\RN)}^{2(1-s)}\|\nabla v_\varepsilon\|_{L^2(\RN)}^{2s}
   \le C\,\varepsilon^{1-s},
\end{equation*}
by \eqref{L2FinalEst} and \eqref{eq:gradexp}, where $c_{N,s}>0$ is the constant of
\cite[Proposition~3.4]{DiNezzaPalatucciValdinoci2012}.

\medskip
\noindent\textbf{Step 1b: the $L^2$ mass.} It remains to bound the second term in \eqref{eq:SimplifiedEstimateofI}. With
$\supp\varphi\subset\{|y|\le r\}$,
\begin{equation*}
   \|v_\varepsilon\|_{L^2(\RN)}^2\le\|w_\varepsilon\|_{L^2(\{|y|\le r\})}^2
   =\varepsilon^2 C(N)\int_0^{r/\varepsilon}f_N(\rho)\,d\rho,
   \qquad f_N(\rho):=\frac{\rho^{N-1}}{(1+\rho^2)^{N-2}}. 
\end{equation*}
Since $f_N$ is bounded near $\rho=0$, we split at $\rho=1$:
\begin{equation*}
   \int_0^{r/\varepsilon}f_N(\rho)\,d\rho
   =\underbrace{\int_0^1 f_N(\rho)\,d\rho}_{=:c_N<\infty}
    +\int_1^{r/\varepsilon}f_N(\rho)\,d\rho,
\end{equation*}
where, on $\rho\ge1$, $f_N(\rho) \leq C\rho^{3-N}$. Hence the tail integral is
$O(1)$ for $N\ge5$, $O(\log(1/\varepsilon))$ for $N=4$, and $O(1/\varepsilon)$ for $N=3$.
Thus
\begin{equation}\label{L2FinalEst}
   \|v_\varepsilon\|_{L^2(\RN)}^2\le C(N,r)\,\varepsilon^{\min\{2,\,N-2\}}\,{\mathcal{M}_N}(\varepsilon),
   \qquad
   {\mathcal{M}_N}(\varepsilon):=\begin{cases}1,&N\ne4,\\|\log\varepsilon|,&N=4.\end{cases}
\end{equation}
Combining \eqref{eq:SimplifiedEstimateofI}, \eqref{eq:reigon2final}, together with \eqref{L2FinalEst} and
\eqref{FractFinalEst},
\begin{equation}\label{finalfinalFrac}
   [u_\varepsilon]_s^2
   \le C\Bigl([v_\varepsilon]_{s,\RN}^2+\|v_\varepsilon\|_{L^2(\RN)}^2\Bigr)
   \le C(N,s,r,r_1)\Bigl(\varepsilon^{\sigma}+\varepsilon^{\min\{2,N-2\}}{\mathcal{M}_N}(\varepsilon)\Bigr),
\end{equation}
where $\sigma:=\begin{cases}2-2s,&N+2s>4,\\1-s,&N=3,\ s\le\frac12.\end{cases}$.\\
We conclude
$[u_\varepsilon]_s^2\to0$ as $\varepsilon\to0$.

\medskip

\noindent\textbf{Step 2: the upper bound.} Under the conformal change of metric relating $\bn$ to the Euclidean ball $B_{\RN}(0,1)$ (for details refer to \cite{OurCCM,DebdipBhaktaKarmakar}), we obtain 
\begin{equation}\label{eq:conformal}
   \int_{\BN}|\gradB u_\varepsilon|^2\,\dV-\lambda\int_{\BN}u_\varepsilon^2\,\dV
   =\int_{B_{\RN}(0,1)}|\nabla v_\varepsilon|^2\,dx-\int_{B_{\RN}(0,1)}g_\lambda(x)\,v_\varepsilon^2\,dx,
\end{equation}
where $g_\lambda(x):=\frac{4\lambda-N(N-2)}{(1-|x|^2)^{2}}$.
The classical Aubin-Talenti
estimates (see \cite[Proposition~3.2]{DebdipBhaktaKarmakar}, \cite[pg. 19]{MousomiUTF}) for $N\ge3$ as $\varepsilon \rightarrow 0$ yield
\begin{align}
   \int_{B_{\RN}(0,1)}|\nabla v_\varepsilon|^2\,dx
     &=\int_{\RN}|\nabla U|^2\,dx+O(\varepsilon^{N-2})=S^{N/2}+O(\varepsilon^{N-2}).
        \label{eq:gradexp}\\[2pt]
   \int_{\BN}|u_\varepsilon|^{\cstar}\,\dV
&=\int_{\RN}U^{\cstar}\,dx+O(\varepsilon^N)=S^{N/2}+O(\varepsilon^N).\notag
\end{align}
\begin{equation}\label{eq:massexact}
   \int_{B_{\RN}(0,1)}g_\lambda(x)\,v_\varepsilon^2\,dx
   =\begin{cases}
      \bigl(4\lambda-N(N-2)\bigr)\,\varepsilon^2\displaystyle\int_{\RN}U^2\,dx
        +o\bigl(\varepsilon^2\bigr), & N\ge5,\\[10pt]
      \bigl(4\lambda-N(N-2)\bigr)\,\omega_{N-1}\,\varepsilon^2\log\tfrac1\varepsilon
        +o\bigl(\varepsilon^2\log\tfrac1\varepsilon\bigr), & N=4,
    \end{cases}
\end{equation}
where $\omega_{N-1}=|\Sph|$.

For $N=3$, as in Step 1b, we get 
\begin{equation}\label{eq:massN3}
   \Bigl|\int_{B_{\RN}(0,1)}g_\lambda\,v_\varepsilon^2\,dx\Bigr|
   \le C(r,\lambda)\,\varepsilon .
\end{equation}

Combining \eqref{eq:conformal}, \eqref{eq:gradexp},
\eqref{eq:massexact} and \eqref{eq:massN3}, the local energy satisfies for
$N\ge3$,
\begin{equation*}
   \|u_\varepsilon\|_\lambda^2
   =S^{N/2}+O\bigl(\varepsilon^{\min\{2,N-2\}}\mathcal M_N(\varepsilon)\bigr).
\end{equation*}
By \eqref{finalfinalFrac}, for $\gamma:=\min\{\sigma,N-2\}$,
\begin{equation*}
   [u_\varepsilon]_s^2\le C\bigl(\varepsilon^{\sigma}+\varepsilon^{\min\{2,N-2\}}\mathcal M_N(\varepsilon)\bigr)
   =O(\varepsilon^{\gamma}).
\end{equation*}
Therefore, for all $\varepsilon>0$ sufficiently small, 
\begin{equation*}
   \frac{\|u_\varepsilon\|_\lambda^2+[u_\varepsilon]_s^2}
        {\bigl(\int_{\BN}|u_\varepsilon|^{\cstar}\,\dV\bigr)^{2/\cstar}}
   =\frac{S^{N/2}+O(\varepsilon^{\gamma})}{S^{(N-2)/2}\bigl(1+O(\varepsilon^N)\bigr)}
   =S+O(\varepsilon^{\gamma}).
\end{equation*}
Letting $\varepsilon\to0$ gives $\Smix\le S$.

\medskip

\textbf{Step 3: the lower bound and non-attainment.}
Recall the conformal change of metric: for $u\in\Hh\setminus\{0\}$, setting
$v=\bigl(\tfrac{2}{1-|x|^2}\bigr)^{(N-2)/2}u$, one has
$v\in H_0^1(B_{\RN}(0,1))$, the Sobolev space of functions on $B_{\RN}(0,1)$
characterized by zero trace on $\partial B_{\RN}(0,1)$, and
\begin{equation}\label{eq:splitloc}
   \|u\|_\lambda^2
   =\int_{B_{\RN}(0,1)}|\nabla v|^2 dx
     +\Bigl(\tfrac{N(N-2)}{4}-\lambda\Bigr)\int_{\BN}u^2\dV,
   \int_{\BN}|u|^{\cstar}\,\dV=\int_{B_{\RN}(0,1)}|v|^{\cstar}dx .
\end{equation}
Using \eqref{eq:splitloc} and \eqref{eq:interp},
\begin{equation*}
   \|u\|_\lambda^2+[u]_s^2
   \ge\int_{B_{\RN}(0,1)}|\nabla v|^2\,dx+c_\lambda\int_{\BN}u^2\,\dV,
   \quad
   c_\lambda:=\frac{N(N-2)}{4}-\lambda+\Bigl(\frac{(N-1)^2}{4}\Bigr)^{s}.
\end{equation*}
As $N\ge3$, we have
$\bigl(\frac{(N-1)^2}{4}\bigr)^{s}\ge1$, while
$\lambda<\frac{(N-1)^2}{4}=\frac{N(N-2)}{4}+\frac14$. Therefore
$c_\lambda>1-\frac14>0$. Extending $v$ by zero outside $B_{\RN}(0,1)$ and using
\eqref{eq:bestSobolevConstantEuclidean} together with the second identity in
\eqref{eq:splitloc}, we obtain for every $u\in\Hh\setminus\{0\}$
\begin{equation}\label{eq:strictlower}
   \|u\|_\lambda^2+[u]_s^2
   \ge S\Bigl(\int_{\BN}|u|^{\cstar}\,\dV\Bigr)^{2/\cstar}
     +c_\lambda\int_{\BN}u^2\,\dV
>S\Bigl(\int_{\BN}|u|^{\cstar}\,\dV\Bigr)^{2/\cstar}.
\end{equation}
Taking the infimum over $u$ gives $\Smix\ge S$, and with Step~2, $\Smix=S$. If
$\Smix$ were attained, it will contradict the strict inequality in  \eqref{eq:strictlower}. This proves Part \textup{(i)}.

\medskip

\noindent\textbf{Proof of (ii).}
Assume $\lambda\le\frac{N(N-2)}{4}$. Then  \eqref{eq:splitloc}, together with the sharp Euclidean Sobolev
inequality \eqref{eq:bestSobolevConstantEuclidean}, implies that for every
$u\in\Hh\setminus\{0\}$
$$
   \frac{\|u\|_\lambda^2}{\|u\|_{\cstar}^2}
   \ge\frac{\int_{B_{\RN}(0,1)}|\nabla v|^2\,dx}
           {\bigl(\int_{B_{\RN}(0,1)}|v|^{\cstar}\,dx\bigr)^{2/\cstar}}\ge S .
$$
Taking the infimum over $u\in\Hh\setminus\{0\}$ yields $\Sloc\ge S$.
Further, for the family $u_\varepsilon$ of Step~1, the estimates \eqref{eq:gradexp}
and \eqref{eq:massexact}-\eqref{eq:massN3} give
$\|u_\varepsilon\|_\lambda^2/\|u_\varepsilon\|_{\cstar}^2\to S$. Hence $\Sloc\le S$.
Combining the two bounds, $\Sloc=S$. Since $\Smix=S$ by part \textup{(i)}, we
conclude $\Sloc=\Smix=S$.

\medskip
\textbf{Proof of (iii).}
For $N\ge4$ and $\frac{N(N-2)}{4}<\lambda<\frac{(N-1)^2}{4}$ one has
$\Sloc<S$ \cite{DebdipBhaktaKarmakar}, and $S=\Smix$ by part \textup{(i)}.
Hence $\Sloc<\Smix$.\qedhere
\end{proof}

\medskip
We now use Theorem~\ref{thm:sobolevConstants}\textup{(i)} to prove the nonexistence
results for the purely critical problem \eqref{mainEqCrit}.

%\section{The purely critical problem}

 \medskip

\begin{proof}[Proof of Theorem~\ref{thm:nonexistCrit}]
(i) Problem \eqref{mainEqCrit} is \eqref{mainEq2} with $\mu=0$, so the first claim is
Theorem~\ref{thm:perturbMu}\textup{(iii)} with $\mu=0$, proved in
Section~\ref{criticalSection}. Moreover, for a solution $u$,
$J_0(u)=\frac1N\int_{\bn}|u|^{\cstar}\dvg$, and $u\notin\mathcal B$ is equivalent to
$J_0(u)>\frac1NS^{N/2}$.

\medskip

(ii) Let $u\in H^1(\bn)\setminus\{0\}$. For $\zeta>0$ and
$J_0(\zeta u)=\frac{\zeta^2}{2}(\|u\|_\lambda^2+[u]_s^2)-\frac{\zeta^{\cstar}}{\cstar}\int_{\bn}|u|^{\cstar}\dvg$, $\zeta u\in\mathcal N_0$ if and only if $\zeta=\zeta_u = \left(\frac{(\|u\|_\lambda^2+[u]_s^2)}{\int_{\bn}|u|^{\cstar}\dvg}\right)^{1/(\cstar-2)}$. And $\zeta_u$ is the maximum point of $\zeta\mapsto J_0(\zeta u)$. Moreover,
\begin{equation*}
   J_0(\zeta_uu)=\Bigl(\frac12-\frac1{\cstar}\Bigr)\zeta_u^2\,(\|u\|_\lambda^2+[u]_s^2)
   =\frac1N\,S_{\lambda,s}(u)^{\frac N2}.
\end{equation*}
Thus 
$\inf_{\mathcal N_0}J_0=\frac1NS_{\lambda,s}^{N/2}=\frac1NS^{N/2}$. Since $S_{\lambda,s}=S$
is not attained by Theorem~\ref{thm:sobolevConstants}\textup{(i)}, we have
$S_{\lambda,s}(u)>S$, hence $J_0(u)>\frac1NS^{N/2}$, for every $u\in\mathcal N_0$. So the
infimum is not attained, and \eqref{mainEqCrit} has no ground state.
\end{proof}

\section{Existence of a Solution in the Subcritical case}\label{subcrticalSection}
\subsection{Variational Formulation}
\begin{definition}
We say that $u \in H^1(\bn)$ is a weak solution to \eqref{mainEq} if it satisfies the variational identity
\begin{small}
\begin{align*}
    &\int_{\bn}\grad_{\bn} u\grad_{\bn} v \dvg+ \frac12\int\int(u(x)-u(y))(v(x)-v(y))\mathcal{K}_s(d(x,y))\dvg(x)\dvg(y)\\
    & \quad - \lambda\int_{\bn}uv \dvg =\int_{\bn}|u|^{p-1}uv \dvg \;\;\forall v \in H^1(\bn).
\end{align*}
\end{small}
\end{definition}

\noindent
Next, we define the energy functional $I: H^1(\bn) \to \mathbb{R}$ by
\begin{align}
    I(u)&= \frac{1}{2} \int_{\bn} |\nabla_{\bn} u|^2 \dvg - \frac{\lambda}{2} \int_{\bn} |u|^2 \dvg +\frac12 [u]_s^2 - \frac{1}{p+1} \int_{\bn} |u|^{p+1} \dvg. \label{EnergyFuncSubcritical}
\end{align}
\noindent
It is straightforward to see that the energy functional $I$ is well-defined using \eqref{PoinSobIneq} and \eqref{1tosEmbedding}. Moreover, its Fr\'echet derivative is given by
\begin{align*}
   I'(u)(v)&= \int_{\bn} \nabla_{\bn} u \cdot \nabla_{\bn} v \dvg - \lambda \int_{\bn} u v \dvg\\
   &+ \frac{1}{2}\int_{\bn}\int_{\bn}(u(x)-u(y))(v(x)-v(y))\mathcal{K}_s(d(x,y))\dvg(x)\dvg(y)\\
    &\quad \quad - \int_{\bn} |u|^{p-1} u v \dvg, \quad \forall v \in H^1(\bn).
\end{align*}
\noindent
Consequently, weak solutions of \eqref{mainEq} correspond to the critical points of the functional $I$.

\smallskip
\noindent
\subsection{Existence of a solution}Now, we establish the existence of a solution to \eqref{mainEq} for values of $p$ within the subcritical range, i.e., $1<p<2^*-1$.

\noindent
\medskip
Consider a function $u \in H^1(\bn)\setminus \{0\}$ and define the function
$$ g(t) = I(tu), \quad t > 0. $$
This function attains a unique maximum at some $t(u) > 0$, and the corresponding scaled function $t(u)u$ belongs to the Nehari manifold $\Lambda$, associated with the energy functional $I$, defined as
\begin{equation}
    \Lambda = \{ u \in H^1(\bn) \setminus \{0\} \mid I'(u)(u) = 0 \}. \label{nehariManifoldSubcritical}
\end{equation}
Explicitly, $t(u)$ is given by
$$ t(u) = \left(\frac{\|u\|_\lambda^2 + [u]_s^2}{\int_{\bn} |u|^{p+1}\dvg}\right)^{\frac{1}{p-1}}. $$

\noindent
Next, we define the minimization level
$$ c^* = \inf_{u \in \Lambda} I(u), $$
and observe that it satisfies the following characterization
\begin{equation*}
    c^{*} = \inf _{u \in H^{1}(\bn) \backslash\{0\}} \sup _{\theta \geq 0} I(\theta u).
\end{equation*}

\noindent
On the other hand, we introduce the set of admissible paths
$$ \Gamma = \left\{ g \in C([0,1], H^{1}(\bn)) \mid g(0) = 0, \;I(g(1)) < 0 \right\}, $$
and define
\begin{equation*}
    c = \inf _{g \in \Gamma} \sup _{t \in[0,1]} I(g(t)).
\end{equation*}
The set $\Gamma$ is nonempty. Indeed, given $u \in H^1(\bn)$, we may define a continuous path $g(t) = tT u$, where the constant $T$ is chosen such that $I(Tu) < 0$.
\begin{definition}
    A weak solution $u$ of \eqref{mainEq} is called a mountain pass weak solution if $I(u)=c$

\end{definition}
\begin{remark}\label{Rem:MtinPassLevel>0}
Moreover, we have $c > 0$. Consider any $u \in H^1(\bn)$. From the definition of $I$, we obtain the estimate
\begin{align}\label{eq1.1}
    I(u) &= \frac{1}{2}\|u\|_\lambda^2 + \frac{1}{2} [u]_s^2 - \frac{1}{p+1} \int_{\bn} |u|^{p+1} \, \dvg \\\notag
         &\geq \frac{1}{2}\|u\|_\lambda^2 - \frac{1}{p+1} \int_{\bn} |u|^{p+1} \, \dvg \\\notag
         &\geq \frac{1}{2}\|u\|_\lambda^2 - \frac{1}{p+1} \|u\|_\lambda^{p+1} S_{\la,p}^{-\frac{p+1}{2}},
\end{align}
where $S_{\la,p}$ is the Sobolev constant defined in \eqref{SubcritPoincare}. Setting
$ r := \left[ \frac{p+1}{4} S_{\la,p}^{\frac{p+1}{2}} \right]^{\frac{1}{p-1}}, $
we obtain the lower bound
$$ \min_{\|u\|_\lambda = r} I(u) \geq \frac{1}{4} \left( \frac{(p+1) S_{\la,p}^{\frac{p+1}{2}}}{4} \right)^{\frac{2}{p-1}}. $$
Since $I(g(1)) < 0$ for any $g \in \Gamma$, it follows that $\|g(1)\|_\lambda \geq r$ using \eqref{eq1.1} for $g(1)$. Consequently, there exists $t_g \in (0,1)$ such that $\|g(t_g)\|_\lambda = r$. This leads to the estimate
$$ \sup _{t \in[0,1]} I(g(t)) \geq I(g(t_g)) \geq \frac{1}{4} \left( \frac{(p+1) S_{\la,p}^{\frac{p+1}{2}}}{4} \right)^{\frac{2}{p-1}}, $$
which implies that $c > 0$.
\end{remark}

%%%%%%%%%%%%%%%%%%%

\medskip

\noindent
We now establish the existence of a critical point for the functional $I$ by employing the mountain pass theorem. Our approach is based on variational methods, where we leverage the compactness of the embedding mentioned in \ref{radial Lemma} to handle the potential losses of compactness.

\begin{lemma}[\cite{MousomiUTF}] \label{radial Lemma}
Let $ H^1_r(\bn) $ denote the subspace of radial functions in $ H^1(\bn) $, that is,
$$
H^1_r(\bn) := \left\{ u \in H^1(\bn) : u \text{ is radial} \right\}.
$$
Then, for every $ 2 < q < 2^* $, the embedding
$$
H^1_r(\bn) \hookrightarrow L^q(\bn)
$$
is compact.
\end{lemma}

\noindent
With this compactness result in hand, we now proceed to establish the following existence and symmetry result for a weak solution of $I$.
We aim to show that $I$ admits at least one critical point with critical value $c$.

Before proceeding, we recall the notion of symmetric rearrangement. For a function $u: \mathbb{B}^N \rightarrow \mathbb{R}^{+} \cup {+\infty}$, the symmetric rearrangement, also known as \textit{Schwarz symmetrization}, $u^*: \mathbb{B}^N \rightarrow \mathbb{R}^{+} \cup {+\infty}$, is defined as the unique function such that for every $\lambda > 0$, there exists $R > 0$ satisfying
\begin{equation*}
\left\{x \in \mathbb{B}^N : u^*(x) > \lambda\right\} = B(0,R),
\end{equation*}
and
\begin{equation*}
\mu\left\{x \in \mathbb{B}^N : u^*(x) > \lambda\right\} = \mu\left\{x \in \mathbb{B}^N : u(x) > \lambda\right\},
\end{equation*}
where $B(0,R)$ denotes the geodesic ball in $\mathbb{B}^N$ centered at $0$ with radius $R$, and $\mu$ denotes the volume measure in $\mathbb{B}^N$.

Thus, $u^*$ is a radial and radially decreasing function whose superlevel sets have the same measure as those of $u$ (refer to \cite{Baernstein1,BrockSolynin}).

\bigskip

\begin{proof}[Proof of Theorem~\ref{ExistCritPt}]

\textbf{\textit{(1)}} Here we prove the existence of a minimizer of $I$ subject to $\Lambda$. Choose a minimizing sequence $\left\{u_n\right\} \in \Lambda$ of $I$. We may assume that each $u_n$ is nonnegative. If it is not, then we can work with $|u_n|$. Indeed,
\begin{align*}
    \||u_n|\|_\lambda^2 + [|u_n|]_s^2 &\leq \|u_n\|_\lambda^2 + [u_n]_s^2 = \int_{\mathbb{B}^N} |u_n|^{p+1} \, \dvg.
\end{align*}

\noindent
The above inequality follows from the properties  
$
|\nabla_{\bn}|u_n||_{\bn} \leq |\nabla_{\bn} u_n|_{\bn}$
and  
$
||u_n(x)| - |u_n(y)|| \leq |u_n(x) - u_n(y)|.
$
Hence if we set 
$$r(t)=I'(t|u_n|)(t|u_n|)= t^2\||u_n|\|_\la^2+t^2[|u_n|]_s^2-t^{p+1}\int_{\bn}|u_n|^{p+1}\dvg,$$ 
thus we have $r(1)\leq 0$ while $r(t)>0$ for $t>0$ and small. So there exist $t_n \in (0,1]$ such that $t_n|u_n|\in \Lambda$. Therefore, we obtain
\begin{align*}
    c^*\leq I(t_n|u_n|)= \frac{t_n^2}{2}\||u_n|\|_\la^2+\frac{t_n^2}{2}[|u_n|]_s^2-\frac{t_n^{p+1}}{p+1}\int_{\bn}|u_n|^{p+1}\dvg\\
    =t_n^{p+1}\left(\frac{1}{2}-\frac{1}{p+1}\right)\int_{\bn}|u_n|^{p+1}\dvg\\
    \leq \left(\frac{1}{2}-\frac{1}{p+1}\right)\int_{\bn}|u_n|^{p+1}\dvg = I(u_n)
\end{align*}

\noindent
which shows $\left\{t_n|u_n|\right\}$ is also a minimizing sequence. Then we may denote $\left\{t_n|u_n|\right\}$ just by $\left\{u_n\right\}$. We also may assume that each $u_n$ is radially symmetric, i.e. $u_n \in H_r^1\left(\mathbb{B}^N\right)$. If it is not, let $u_n^* \in H_r^1\left(\mathbb{B}^N\right)$ be the symmetric decreasing rearrangement of $u_n$ . From \cite[Theorem 3]{Beckner92} (also refer \cite{Baernstein1}), and Lemma~\ref{RadDecreasingLemma},  one has

\begin{equation*}
\int_{\mathbb{B}^N}\left|(-\Delta_{\bn})^{s / 2} u_n^*\right|^2 \dvg \leqslant \int_{\mathbb{B}^N}\left|(-\Delta_{\bn})^{s / 2} u_n\right|^2 \dvg
\end{equation*}

\begin{equation*}
    \int_{\bn}|\nabla_{\bn} u_n^*|^2 \dvg\leq \int_{\bn}|\nabla_{\bn} u_n|^2 \dvg,
\end{equation*}
and it is well known that
\begin{equation*}
\int_{\mathbb{B}^N}\left(u_n^*\right)^q\dvg=\int_{\mathbb{B}^N}  u_n^q \dvg, \text{ \;\;for any $q>1$.}
\end{equation*}
%simply foolows from layer cake representaion
so once more we have $I\left(u_n^*\right) \leqslant I\left(u_n\right)$. Consequently, there exists $t_n>0$ such that $\left\{t_nu_n^*\right\}$ is a minimizing sequence as before. We again denote  $\left\{t_nu_n^*\right\}$ by $\left\{u_n\right\}$.

\noindent
The boundedness of $\left\{u_n\right\}$ in $H^1\left(\mathbb{B}^N\right)$ easily follows from $I'\left(u_n\right)(u_n)=0$ and $I(u_n)\rightarrow c^*$. As a consequence, up to a subsequence, we have the following convergence
\begin{equation}
u_n \stackrel{n \rightarrow \infty}{\longrightarrow}u_0\in H_r^1\left(\mathbb{B}^N\right)\left\{\begin{array}{l}
\text { strongly in } L^{p+1}\left(\bn \right) \text{ for } 1<p<2^* -1,\\
\text { weakly in } H^1\left(\bn \right) .
\end{array}\right. \label{convergnceOfMineq}
\end{equation}

\noindent
Note that $$I(u_n)=\left(\frac{1}{2}-\frac{1}{p+1}\right)\int_{\bn}|u_n|^{p+1}\dvg.$$
Thus, \eqref{convergnceOfMineq}, and Remark~\ref{Rem:MtinPassLevel>0} imply  $u_0 \not\equiv 0$.\\

\noindent
Furthermore, by Ekeland’s variational principle (see \cite[Corollary 3.4]{IE}), we can select the minimizing sequence $ \{u_n\} \subset \Lambda \cap H^1_r(\mathbb{B}^N) $ such that
\begin{equation}
I^{\prime}\left(u_n\right)[v]=\theta_n G^{\prime}\left(u_n\right)[v]+o(1)\|v\|_{\lambda} \quad \forall v \in H_r^1\left(\bn\right). \label{a.4}
\end{equation}
where for all $n \in \mathbb{N}, \theta_n \in \mathbb{R} $ is the Lagrange multiplier and $ G(u) := I^{\prime}(u)[u] $. Moreover, we have $ G(u_n) = 0 $ for all $ n \in \mathbb{N} $, and thus \eqref{a.4} implies
\begin{equation*}
0=G\left(u_n\right)=I^{\prime}\left(u_n\right)\left[u_n\right]=\theta_n G^{\prime}\left(u_n\right)\left[u_n\right]+o(1)\left\|u_n\right\|_{\la}.\end{equation*}
\noindent
Since $ \|u_n\|_{\lambda} $ is bounded and $ G^{\prime}(u_n)[u_n] \leq \tilde{C} < 0 $  on $\Lambda$ for some constant $\tilde{C}$, it follows that $ \theta_n = o(1) $ as $ n \to \infty $.
Next, choosing $ v = u_0 $ in \eqref{a.4} and using \eqref{convergnceOfMineq}, we deduce that $ u_0 \in \Lambda $.
Moreover, applying again \eqref{convergnceOfMineq}, we obtain
\begin{equation*}
c^* \leq I\left(u_0\right) = \lim_{n \rightarrow \infty}\left(\frac{1}{2}-\frac{1}{p+1}\right)\int_{\bn}|u_n|^{p+1}=c^*.
\end{equation*}
This shows that the desired minimizing function is $ u_0 $. Moreover, we may assume $ u_0 \geq 0 $, since otherwise we can consider $ t|u_0| $ as a minimizer by following the same reasoning as above.

\medskip
\noindent
\textbf{Proof of }\textit{(2)}:
To prove this part it is enough to show that $c=c^*$. Given any $u \in \Lambda$, we define a path $g_u$ as $g_u(t) = tTu$, where $I(Tu) < 0$. This ensures that $g_u \in \Gamma$, which immediately implies $c \leqslant c^*$.

\medskip
\noindent
The reverse inequality follows from the fact that, for any $g \in \Gamma$, there exists $t \in (0,1)$ such that $g(t) \in \Lambda$.

\noindent
To justify this claim, observe that if $I'(u) u \geqslant 0$, then we obtain
\begin{equation} \label{eq:1.1a}
I(u) \geq \frac{1}{2} \int_{\mathbb{B}^N} |u|^{p+1} - \frac{1}{p+1} \int_{\mathbb{B}^N} |u|^{p+1} \geqslant 0, \quad \text{as } p > 1.
\end{equation}
Thus, assuming $I'(g(t)) g(t) \geq 0$ for all $t \in (0,1]$, it follows that $I(g(t)) \geq 0$ for all $t \in (0,1]$, contradicting the fact that $I(g(1)) < 0$.

\noindent
Moreover, for any $g \in \Gamma$, we have $I(g(1)) < 0$, which leads to the estimate
$$ \frac{1}{2} \|g(1)\|_\lambda^2 + \frac{1}{2} [g(1)]_s^2 - \frac{1}{p+1} \int_{\mathbb{B}^N} |g(1)|^{p+1} \dvg < 0. $$
Thus, we deduce that
$$ I'(g(1))(g(1)) < \left( \frac{2}{p+1} - 1 \right) \int_{\mathbb{B}^N} |g(1)|^{p+1} \dvg < 0. $$

\noindent
From this, we further obtain
\begin{equation} \label{eq:1a}
\begin{aligned}
    0 > I'(g(1))(g(1)) &= \|g(1)\|_\lambda^2 + [g(1)]_s^2 - \int_{\mathbb{B}^N} |g(1)|^{p+1} \dvg \\
    & \geq \|g(1)\|_\lambda^2 - \|g(1)\|_\lambda^{p+1} S_{\la,p}^{-\frac{p+1}{2}}.
\end{aligned}
\end{equation}
\noindent
Setting $\tilde{r} = \left( \frac{1}{4} S_{\la,p}^{\frac{p+1}{2}} \right)^{\frac{1}{p-1}}$, we conclude that $\|g(1)\|_\lambda > \tilde{r}$. Since $g(0) = 0$, there exists some $\tilde{t}_g \in (0,1)$ such that $\|g(\tilde{t}_g)\|_\lambda = \tilde{r}$.

\noindent
Applying \eqref{eq:1a}, we obtain
$$ \min_{\|u\|_\lambda = \tilde{r}} I'(u)(u) \geq \frac{3}{4} \left( \frac{1}{4} S_{\la,p}^{\frac{p+1}{2}} \right)^{\frac{2}{p-1}} > 0, \quad \text{for all } u \in H^1(\mathbb{B}^N). $$
In particular, $I'(g(\tilde{t}_g))(g(\tilde{t}_g)) > 0$. By combining this observation with \eqref{eq:1.1a}, we establish our claim, which immediately yields $c \geq c^*$. This completes the proof of $\textit{(2)}$.

\medskip
\noindent
 \textbf{Proof of} \textit{(3)}: Suppose that there exists a minimizer $u \in \Lambda$ of $I$, subject to $\Lambda$, which is not radially symmetric up to a translation. Now, let $u^*$ be the symmetric rearrangement of $u$. Then, from \cite[Theorem 3]{Beckner92}, and Lemma~\ref{RadDecreasingLemma}, one has a strict inequality
\begin{equation*}
\int_{\mathbb{B}^N} |(-\Delta_{\bn})^{s/2} u^*|^2 \, \dvg < \int_{\mathbb{B}^N} |(-\Delta_{\bn})^{s/2} u|^2 \, \dvg.
\end{equation*}
As in the proof of $\textit{(1)}$,  there exists $t_0$ such that $t_0u^*\in \Lambda$.
Now consider
\begin{align*}
    I(t_0u^*)&= \frac{t_0^2}{2}\|u^*\|_\la^2+\frac{t_0^2}{2}[u^*]_s^2-\frac{t_0^{p+1}}{p+1}\int_{\bn}|u^*|^{p+1}\dvg\\
   & < \frac{t_0^2}{2}\|u\|_\la^2+\frac{t_0^2}{2}[u]_s^2-\frac{t_0^{p+1}}{p+1}\int_{\bn}|u|^{p+1}\dvg\\
    &\leq \left(\frac{t_0^2}{2}-\frac{t_0^{p+1}}{p+1}\right)\int_{\bn}|u|^{p+1}\dvg\\
    &\leq \left(\frac{1}{2}-\frac{1}{p+1}\right)\int_{\bn}|u|^{p+1}\dvg =I(u).
\end{align*}
This contradicts the definition of $ c^*$, and thus concludes the proof of Theorem~\ref{ExistCritPt}.\qedhere
\end{proof}

\section{The perturbed critical problem}\label{criticalSection}
\noindent
The aim of this section is to establish the existence of nontrivial nonnegative solutions
to the problem \eqref{mainEq2} for $\mu>0$, and to prove Theorem~\ref{CriticalExponentMainThm}
and Theorem~\ref{thm:perturbMu}. To this end, we adapt the variational techniques to the
mixed setting. Specifically, we investigate the critical points of the associated
energy functional $J_\mu : H^1(\bn) \to \mathbb{R}$, defined by
\begin{align} \label{energyFunctionalCritical}
J_\mu(u) &= \frac{1}{2}\int_{\bn}|\nabla_{\bn}u|^2\dvg+\frac12[u]_s^2
- \frac{\lambda}{2} \int_{\bn} |u|^2  \dvg\notag\\
&- \frac{1}{2^*} \int_{\bn} |u|^{2^*} \dvg  -\frac{\mu}{p+1}\int_{\bn}|u|^{p+1}\dvg .
\end{align}
By \eqref{1tosEmbedding} and the Poincar\'e--Sobolev inequality \eqref{SubcritPoincare},
$J_\mu$ is well defined on $H^1(\bn)$. To obtain nonnegative solutions, we also use the
truncated functional
\begin{equation}\label{energyFunctionalCriticalPlus}
J_\mu^+(u)=\frac12\Bigl(\|u\|_\lambda^2+[u]_s^2\Bigr)
-\frac{1}{2^*}\int_{\bn}(u^+)^{2^*}\dvg-\frac{\mu}{p+1}\int_{\bn}(u^+)^{p+1}\dvg ,
\, u^+:=\max\{u,0\},
\end{equation}
which coincides with $J_\mu$ on nonnegative functions.

\begin{remark}\label{symmetricCriticalityRemark}
Observe that, by the principle of symmetric criticality \cite{symmetricPrinciple}, a
critical point of $J_\mu$ or of $J_\mu^+$ in $H^1_r(\bn)$ is a critical point in
$H^1(\bn)$, since both functionals are invariant under rotations, which are isometries of
$\bn$.
\end{remark}

The equation \eqref{mainEq2} is the Euler--Lagrange equation of $J_\mu$. The functional
$J_\mu$ is Fr\'echet differentiable on $H_r^1(\bn)$, and for every $\varphi \in H_r^1(\bn)$,
\begin{align*}
    \langle J_\mu'(u), \varphi \rangle &= \int_{\bn} \nabla_{\bn} u \cdot \nabla_{\bn} \varphi \, \dvg- \lambda \int_{\bn} u \varphi \, \dvg  \\
    &\quad+\frac{1}{2}\int_{\bn}\int_{\bn} (u(x) - u(y)) (\varphi(x) - \varphi(y)) \, \mathcal{K}_s(d(x,y)) \, \dvg(x)\dvg(y) \\
    &\quad - \int_{\bn} |u|^{2^*-2} u \varphi \, \dvg - \mu\int_{\bn} |u|^{p-1} u \varphi \, \dvg.
\end{align*}
In particular, testing with $\varphi=u$,
\begin{equation}\label{eq:Jprimeuu}
   \langle J_\mu'(u),u\rangle=\|u\|_\lambda^2+[u]_s^2-\int_{\bn}|u|^{2^*}\dvg
   -\mu\int_{\bn}|u|^{p+1}\dvg .
\end{equation}
Also,
\begin{equation}\label{eq:JprimePlusuu}
   \langle (J_\mu^+)'(u),u\rangle=\|u\|_\lambda^2+[u]_s^2-\int_{\bn}(u^+)^{2^*}\dvg
   -\mu\int_{\bn}(u^+)^{p+1}\dvg .
\end{equation}
Thus, by Remark~\ref{symmetricCriticalityRemark}, the critical points of $J_\mu$ on
$H^1_r(\bn)$ are weak solutions of \eqref{mainEq2}. To identify such critical points, we
use the Mountain Pass Theorem for $J_\mu^+$. However, the classical Mountain Pass Theorem
cannot be directly applied in this context due to the lack of compactness in the embedding
$ H_r^1(\bn) \hookrightarrow L^{2^*}(\mathbb{B}^N) $. This lack of compactness implies that
the functional $J_\mu^+$ does not satisfy the Palais–Smale condition globally.
Nevertheless, compactness can still be recovered within a suitable range of energy levels,
specifically those below a threshold determined by the best Sobolev constant $S_{\lambda,s}$
associated with the mixed local–nonlocal operator defined in \eqref{eq:SK}.

\smallskip
To circumvent these limitations, we employ a generalized version of the Mountain Pass
Theorem that does not require the Palais–Smale condition to hold globally (see \cite{AR},
\cite[Theorem 2.2]{BrezisNirenberg1983}). This approach hinges on verifying that the functional $ J_\mu^+ $
possesses an appropriate geometric structure, specifically, the so-called mountain pass
geometry on the radial Sobolev space $ H_r^1(\bn) $.

\subsection{Existence of a solution in the critical case with a lower order perturbation}We begin by establishing that the functional $J_\mu^+$ indeed exhibits the mountain pass geometry stated in \cite[Theorem 2.2]{BrezisNirenberg1983}.

\begin{proposition}\label{BdryMoutainpass}
Assume $ \lambda < \frac{(N-1)^2}{4} $, $ 1 < p < 2^* - 1 $ and $\mu>0$. Then there exist
constants $ \rho > 0 $ and $ \alpha > 0 $ such that $ J_\mu^+(u) \geq \alpha $ for every
$ u \in H_r^1(\bn) $ with $ \|u\|_{\lambda} = \rho $.
\end{proposition}

\begin{proof}
Let $ u \in H_r^1(\bn) $. By \eqref{SubcritPoincare} with the exponents $p$ and $2^*-1$,
\begin{comment}
and since $0\le u^+\le|u|$,
\begin{equation*}
   \int_{\bn}(u^+)^{p+1}\dvg\le S_{\lambda,p}^{-\frac{p+1}{2}}\|u\|_\lambda^{p+1},
   \qquad
   \int_{\bn}(u^+)^{2^*}\dvg\le S_{\lambda,2^*-1}^{-\frac{2^*}{2}}\|u\|_\lambda^{2^*}.
\end{equation*}
Since $[u]_s^2\ge0$,
\end{comment}
\begin{align*}
    J_\mu^+(u) &= \frac{1}{2} \left( \|u\|^2_\lambda + [u]_s^2 \right) - \frac{1}{2^*} \int_{\mathbb{B}^N} (u^+)^{2^*} \, \dvg - \frac{\mu}{p+1} \int_{\mathbb{B}^N} (u^+)^{p+1} \, \dvg \\
    &\geq \frac{1}{2} \|u\|_\lambda^2 - C_1 \|u\|_\lambda^{2^*} - \mu C_2 \|u\|_\lambda^{p+1}
    = \frac12\|u\|_\lambda^2 \left( 1 - 2C_1 \|u\|_\lambda^{2^* - 2} - 2\mu C_2 \|u\|_\lambda^{p - 1} \right).
\end{align*}
Since $2^* > p+1 > 2$, we can
choose $\rho>0$ so small that $2C_1\rho^{2^*-2}+2\mu C_2\rho^{p-1}\le\frac12$. Then, for
every $u\in H_r^1(\bn)$ with $\|u\|_\lambda=\rho$,
$$
    J_\mu^+(u) \geq \frac14\rho^2 =: \alpha > 0 .
$$\qedhere
\end{proof}

\begin{proposition}\label{OutsideBallMountainPass}
Let $ \lambda < \frac{(N-1)^2}{4} $, $1 < p < 2^* - 1$ and $\mu>0$, and let
$u_0\in H_r^1(\bn)\setminus\{0\}$ with $u_0\ge0$ a.e. in $\bn$. Then there is $\zeta_0>0$
such that
\begin{equation*}
    e := \zeta_0 u_0
\end{equation*}
satisfies $e \geq 0$ a.e. in $\mathbb{B}^N$, $\|e\|_\lambda > \rho$ and $J_\mu^+(e) < \alpha$,
where $\rho$ and $\alpha$ are the constants of Proposition~\ref{BdryMoutainpass}.
\end{proposition}

\begin{proof}
Since $u_0\ge0$, we have $J_\mu^+(\zeta u_0)=J_\mu(\zeta u_0)$ and, for $\zeta > 0$,
\begin{align*}
    J_\mu(\zeta u_0) &= \frac{\zeta^2}{2} \left( \|u_0\|_\lambda^2 + [u_0]_s^2 \right) - \frac{\zeta^{2^*}}{2^*} \int_{\mathbb{B}^N} |u_0|^{2^*} \, \dvg - \frac{\mu\,\zeta^{p+1}}{p+1} \int_{\mathbb{B}^N} |u_0|^{p+1} \, \dvg \\
    &\le \frac{\zeta^2}{2} \left( \|u_0\|_\lambda^2 + [u_0]_s^2 \right) - \frac{\zeta^{2^*}}{2^*} \int_{\mathbb{B}^N} |u_0|^{2^*} \, \dvg.
\end{align*}
Since $\int_{\bn}|u_0|^{2^*}\dvg>0$ and $2^*>2$, the
right hand side tends to $-\infty$ as $\zeta\to\infty$. Also
$\|\zeta u_0\|_\lambda=\zeta\|u_0\|_\lambda\to\infty$. Hence
there is $\zeta_0>0$ with $\|\zeta_0u_0\|_\lambda>\rho$ and $J_\mu^+(\zeta_0u_0)<\alpha$.
Finally $e=\zeta_0u_0\ge0$ a.e.
\end{proof}

\medskip \noindent We are now in a position to prove Theorem~\ref{CriticalExponentMainThm}.
\begin{proof}[\textbf{Proof of Theorem~\ref{CriticalExponentMainThm}}]
Let $u_0$ be the function in \eqref{CompactnessAssumptionOnFunctional}, and let $e=\zeta_0u_0$
be given by Proposition~\ref{OutsideBallMountainPass}. By
Propositions~\ref{BdryMoutainpass} and \ref{OutsideBallMountainPass}, and since
$J_\mu^+(0)=0<\alpha$, the functional $J_\mu^+$ has the mountain pass geometry required in
\cite[Theorem 2.2]{BrezisNirenberg1983}. We define the mountain pass level
\begin{equation*}
    m = \inf_{P \in \mathcal{P}} \sup_{v \in P([0,1])} J_\mu^+(v),
    \qquad
    \mathcal{P} = \left\{ P \in C([0,1], H_r^1(\bn)) : P(0) = 0, \, P(1) = e \right\}.
\end{equation*}

\medskip
\noindent
\textbf{Step 1:} $\alpha \leq m < \frac{1}{N} S_{\lambda,s}^{N/2}$.

\smallskip
\noindent
\textit{Proof of Step 1:}
Let $P \in \mathcal{P}$. The map $t \mapsto \| P(t) \|_{\lambda}$ is continuous on
$[0,1]$, with $\| P(0) \|_{\lambda} = 0 < \rho$ and $\| P(1) \|_{\lambda} = \| e \|_{\lambda} > \rho$.
So there is $\bar{t} \in (0,1)$ with $\| P(\bar{t}) \|_{\lambda} = \rho$, and by
Proposition~\ref{BdryMoutainpass},
$$
\sup_{v \in P([0,1])} J_\mu^+(v) \geq J_\mu^+(P(\bar{t})) \geq \inf_{\| v \|_{\lambda} = \rho} J_\mu^+(v),
$$
which immediately implies that $m \geq  \alpha$.

The path $t \mapsto t \zeta_0 u_0$, $t \in [0,1]$, belongs to $\mathcal{P}$, and
$J_\mu^+(\zeta u_0)=J_\mu(\zeta u_0)$ since $u_0\ge0$. By
\eqref{CompactnessAssumptionOnFunctional},
$$
m\le\sup_{t\in[0,1]} J_\mu^+(t\zeta_0u_0) \leq \sup_{\zeta \geq 0} J_\mu(\zeta u_0) < \frac{1}{N} S_{\lambda,s}^{N/2}.
$$
This proves Step 1.

By \cite[Theorem 2.2]{BrezisNirenberg1983}, there is a sequence $\{u_j\}\subset H_r^1(\bn)$ such that
\begin{equation} \label{Jc_seq}
    J_\mu^+(u_j) \to m
\end{equation}
and
\begin{equation} \label{Jprime_seq}
    \sup \left\{ \left| \langle (J_\mu^+)'(u_j), \varphi \rangle \right|
    : \varphi \in H_r^1(\bn), \, \| \varphi \|_{\lambda} = 1 \right\} \to 0
\end{equation}
as $j \to +\infty$.

\bigskip
\noindent
\textbf{Step 2:} The sequence $\{u_j\}$ is bounded in $H_r^1(\bn)$.

\smallskip
\noindent
\textit{Proof of Step 2:}
By \eqref{Jc_seq} and \eqref{Jprime_seq}, there is $\tau > 0$ such that for all
$j \in \mathbb{N}$, $| J_\mu^+(u_j) | \leq \tau$ and
$|\langle (J_\mu^+)'(u_j), u_j \rangle| \leq \tau\|u_j\|_\lambda$. Hence
\begin{equation} \label{Jc_estimate}
    J_\mu^+(u_j) - \frac{1}{p+1} \langle (J_\mu^+)'(u_j), u_j \rangle
    \leq \tau (1 + \| u_j \|_{\lambda}).
\end{equation}
On the other hand, by \eqref{energyFunctionalCriticalPlus} and \eqref{eq:JprimePlusuu}, the terms
with $\mu$ cancel, and
\begin{align} \label{Jc_lower_bound}
    &J_\mu^+(u_j) - \frac{1}{p+1} \langle (J_\mu^+)'(u_j), u_j \rangle \notag\\
    &=\left( \frac{1}{2} - \frac{1}{p+1} \right) \left( \|u_j\|_{\lambda}^2 + [u_j]_s^2 \right)
    + \left( \frac{1}{p+1} - \frac{1}{2^*} \right) \int_{\bn} (u_j^+)^{2^*}\dvg \geq \frac{p-1}{2(p+1)} \| u_j \|_{\lambda}^2,
\end{align}
where the inequality holds since $2 < p+1 < 2^*$. Thus, by comparing \eqref{Jc_estimate} and \eqref{Jc_lower_bound}, we conclude that $\| u_j \|_{\lambda}$ is bounded, thereby proving Step 2.

\bigskip
\noindent
\textbf{Step 3.} \textit{The weak limit $u_\infty \in H_r^1(\mathbb{B}^N)$ of $\{u_j\}$ is a
critical point of $J_\mu^+$.}

\medskip
\noindent
\textit{Proof of Step 3.} By Step 2 and reflexivity, there exists
$ u_\infty \in H_r^1(\mathbb{B}^N) $ such that, up to a subsequence,
$u_j \rightharpoonup u_\infty$ weakly in $H_r^1(\mathbb{B}^N)$. Hence, for every
$ \varphi \in H_r^1(\mathbb{B}^N) $,
\begin{align*}
    &\int_{\bn}\nabla_{\bn}u_j \cdot\nabla_{\bn}\varphi \dvg - \lambda\int_{\bn}u_j \varphi \dvg\\
    &\rightarrow  \int_{\bn}\nabla_{\bn}u_\infty\cdot \nabla_{\bn}\varphi \dvg - \lambda\int_{\bn}u_\infty \varphi  \dvg ,
\end{align*}
as $ j \to +\infty $.

\noindent
Moreover, for fixed $\varphi$ the nonlocal term below is a continuous linear functional of
$u$ on $H_r^1(\bn)$, by the Cauchy--Schwarz inequality and \eqref{1tosEmbedding}, so
\begin{align*}
\int_{\bn}\int_{\bn} (u_j(x) - u_j(y))(\varphi(x) - \varphi(y)){\mathcal{K}_s}(d(x,y)) \dvg(x) \dvg(y)\\\notag
\rightarrow \int_{\bn}\int_{\bn}(u_{\infty}(x) - u_{\infty}(y))(\varphi(x) - \varphi(y)){\mathcal{K}_s}(d(x,y)) \dvg(x) \dvg(y)
\end{align*}
for any $ \varphi \in H_r^1(\mathbb{B}^N) $.

\noindent
By the Sobolev inequality \eqref{PoinSobIneq}, $\{u_j\}$ is bounded in
$L^{2^*}(\mathbb{B}^N)$. By Lemma~\ref{radial Lemma}, up to a further subsequence,
\begin{equation}
u_j \to u_\infty \quad \text{strongly in} \quad L^\nu(\mathbb{B}^N) \quad \text{for any} \quad \nu \in (2, 2^*),\label{SubcrittTermConvergence}
\end{equation}
and $u_j \to u_\infty$ almost everywhere in $\mathbb{B}^N$. In particular, since
$p+1\in(2,2^*)$ and $|a^+-b^+|\le|a-b|$,
$\int_{\bn}(u_j^+)^{p}\varphi\dvg\to\int_{\bn}(u_\infty^+)^{p}\varphi\dvg$, and by
Lemma~\ref{nonLinearTermConvergence},
$\int_{\bn}(u_j^+)^{2^*-1}\varphi\dvg\to\int_{\bn}(u_\infty^+)^{2^*-1}\varphi\dvg$.

By \eqref{Jprime_seq}, $\langle (J_\mu^+)'(u_j), \varphi \rangle\to0$ for every
$\varphi \in H_r^1(\mathbb{B}^N)$. Passing to the limit term by term,
$$
\begin{aligned}
0 =\,& \int_{\mathbb{B}^N} \nabla_{\mathbb{B}^N} u_\infty\cdot \nabla_{\mathbb{B}^N} \varphi \, \dvg - \lambda \int_{\mathbb{B}^N} u_\infty \varphi \, \dvg \\
&+ \frac{1}{2}\int_{\bn}\int_{\bn} (u_\infty(x) - u_\infty(y)) (\varphi(x) - \varphi(y)) \mathcal{K}_s(d(x,y)) \, \dvg(x) \dvg(y) \\
&- \int_{\mathbb{B}^N} (u_\infty^+)^{2^*-1} \varphi \, \dvg - \mu\int_{\mathbb{B}^N} (u_\infty^+)^{p} \varphi \, \dvg
\end{aligned}
$$
for every $ \varphi \in H_r^1(\mathbb{B}^N) $. So $u_\infty$ is a critical point of $J_\mu^+$
on $H_r^1(\bn)$, and by Remark~\ref{symmetricCriticalityRemark} on $H^1(\bn)$.

\bigskip
\noindent
\textbf{Step 4.} $u_{\infty}$ is nontrivial.

\smallskip
\noindent
\textit{Proof of Step 4:} Assume, by contradiction, that $u_{\infty} \equiv 0$. Then
\eqref{SubcrittTermConvergence} with $\nu=p+1$ gives
$\int_{\bn}(u_j^+)^{p+1}\dvg\to0$. By Step 2 and \eqref{Jprime_seq},
$\langle (J_\mu^+)'(u_j), u_j \rangle\to0$, so \eqref{eq:JprimePlusuu} gives
\begin{equation}\label{eq1.2_new}
\|u_j\|_{\lambda}^2 + [u_j]_s^2 - \int_{\bn} (u_j^+)^{2^*} \dvg \to 0
\end{equation}
as $j \to +\infty$. By Step 2 and \eqref{1tosEmbedding}, $\|u_j\|_{\lambda}^2+[u_j]_s^2$
is bounded. So, up to a subsequence,
\begin{equation}\label{convEqn_new}
\|u_j\|_{\lambda}^2 + [u_j]_s^2 \to L
\end{equation}
for some $L \in [0, +\infty)$, and by \eqref{eq1.2_new},
\begin{equation}\label{CongEq2_new}
\int_{\bn} (u_j^+)^{2^*} \dvg \to L.
\end{equation}
By \eqref{Jc_seq}, \eqref{convEqn_new}, \eqref{CongEq2_new} and
$\int_{\bn}(u_j^+)^{p+1}\dvg\to0$,
\begin{equation}\label{mountainPassCriticalPtequality_new}
m = \lim_{j\to+\infty} J_\mu^+(u_j)
= \frac12L-\frac{1}{2^*}L = \frac{1}{N} L .
\end{equation}
By Step 1, $m \geq \alpha > 0$, so $L>0$. By \eqref{eq:SK} and $0\le u_j^+\le|u_j|$,
$$
\|u_j\|_{\lambda}^2 + [u_j]_s^2 \geq S_{\lambda,s} \|u_j\|_{L^{2^*}(\bn)}^2
\ge S_{\lambda,s}\Bigl(\int_{\bn}(u_j^+)^{2^*}\dvg\Bigr)^{2/2^*},
$$
and letting $j\to\infty$, by \eqref{convEqn_new} and \eqref{CongEq2_new},
$L \geq S_{\lambda,s} L^{2/2^*}$. This gives
$L\ge S_{\lambda,s}^{N/2}$, and by \eqref{mountainPassCriticalPtequality_new},
$
m \geq \frac{1}{N} S_{\lambda,s}^{N/2},
$
which contradicts Step 1. Thus $u_{\infty} \not\equiv 0$.

\bigskip
\noindent
\textbf{Step 5.} $u_\infty\ge0$, and $u_\infty$ is a weak solution of \eqref{mainEq2}.

\smallskip
\noindent
\textit{Proof of Step 5:} By Step 3, for every $\varphi\in H^1(\bn)$ with $\varphi\ge0$,
\begin{align*}
&\langle u_\infty,\varphi\rangle_\lambda
+\frac12\int_{\bn}\int_{\bn}(u_\infty(x)-u_\infty(y))(\varphi(x)-\varphi(y))\Ks(d(x,y))\dvg(x)\dvg(y)\\
&\hspace{3cm}=\int_{\bn}\bigl((u_\infty^+)^{2^*-1}+\mu(u_\infty^+)^{p}\bigr)\varphi\,\dvg\ \ge0 ,
\end{align*}
since $\mu>0$. By Lemma~\ref{weakMaxLemma}, $u_\infty\ge0$. Hence $u_\infty^+=u_\infty$, and
$u_\infty$ is a weak solution of \eqref{mainEq2}. This completes the proof of
Theorem~\ref{CriticalExponentMainThm}.
\end{proof}

\subsection{Verification of the compactness condition}

We now verify \eqref{CompactnessAssumptionOnFunctional} using the family $u_\varepsilon$
of \eqref{eq:concAubinTalentiBubbles} and the estimates established in the proof of
Theorem~\ref{thm:sobolevConstants}.

\begin{lemma}\label{lem:verifyCompactness}
Let $N\ge3$, $s\in(0,1)$, $\lambda<\frac{(N-1)^2}{4}$, $\mu>0$, and let $p$ satisfy
\eqref{eq:prange}. Then for all sufficiently small $\varepsilon>0$ the function
$u_0=u_\varepsilon$ of \eqref{eq:concAubinTalentiBubbles} satisfies
\eqref{CompactnessAssumptionOnFunctional}.
\end{lemma}

\begin{proof}
The function $u_\varepsilon$ introduced in Theorem~\ref{thm:sobolevConstants} is smooth, radial, nonnegative, nontrivial and
compactly supported, so $u_\varepsilon\in H^1_r(\bn)$. By
Theorem~\ref{thm:sobolevConstants}\textup{(i)} the threshold in
\eqref{CompactnessAssumptionOnFunctional} is $\frac1N\Smix^{N/2}=\frac1NS^{N/2}$.
Write
\begin{equation*}
   Q_\varepsilon=\|u_\varepsilon\|_\lambda^2+[u_\varepsilon]_s^2,\qquad
   A_\varepsilon=\int_{\bn}|u_\varepsilon|^{\cstar}\dvg,\qquad
   B_\varepsilon=\int_{\bn}|u_\varepsilon|^{p+1}\dvg .
\end{equation*}

\smallskip
Combining the estimates for
$u_\varepsilon$ established in Step 1 and 2 of the proof of
Theorem~\ref{thm:sobolevConstants}, we get
\begin{equation}\label{eq:QAexp}
   Q_\varepsilon=S^{\frac N2}+R_\varepsilon,
   \qquad
   |R_\varepsilon|\le C\bigl(\varepsilon^{\sigma}
   +\varepsilon^{\min\{2,N-2\}}\mathcal M_N(\varepsilon)\bigr),
   \qquad
   A_\varepsilon=S^{\frac N2}+O(\varepsilon^N),
\end{equation}
where $\sigma$ is as in \eqref{finalfinalFrac} and $\mathcal M_N$ is as in \eqref{L2FinalEst}.

\smallskip
On $\{|x|<r/2\}$, we have $\varphi=1$, so by
\eqref{eq:concAubinTalentiBubbles}
$|u_\varepsilon|^{p+1}\dvg=w_\varepsilon^{p+1}\bigl(\frac{2}{1-|x|^2}\bigr)^{\beta}dx$
with
\begin{equation*}
   \beta:=N-\frac{(N-2)(p+1)}{2}>0 ,
\end{equation*}
because $p+1<\cstar$. As $\frac{2}{1-|x|^2}\ge2$, the substitution
$x=\varepsilon y$ gives, for $\varepsilon\le r/2$,
\begin{equation}\label{eq:Bbelow}
   B_\varepsilon\ge\int_{|x|<r/2}w_\varepsilon^{p+1}dx
   =\varepsilon^{\beta}\int_{|y|<r/(2\varepsilon)}U^{p+1}dy
   \ge c_0\,\varepsilon^{\beta},\qquad c_0:=\int_{|y|<1}U^{p+1}dy>0 .
\end{equation}
Moreover, by \eqref{SubcritPoincare}
$B_\varepsilon\le S_{\lambda,p}^{-\frac{p+1}{2}}\|u_\varepsilon\|_\lambda^{p+1}$, which is
bounded independently of $\varepsilon$ by \eqref{eq:QAexp}.

For $\zeta\ge0$ put
\begin{equation*}
   g_\varepsilon(\zeta):=J_\mu(\zeta u_\varepsilon)
   =\frac{\zeta^2}{2}Q_\varepsilon-\frac{\zeta^{\cstar}}{\cstar}A_\varepsilon
    -\frac{\mu\,\zeta^{p+1}}{p+1}B_\varepsilon,
   \qquad
   h_\varepsilon(\zeta):=\frac{\zeta^2}{2}Q_\varepsilon
    -\frac{\zeta^{\cstar}}{\cstar}A_\varepsilon ,
\end{equation*}
so that $g_\varepsilon=h_\varepsilon-\frac{\mu\,\zeta^{p+1}}{p+1}B_\varepsilon$.

The supremum of $g_\varepsilon$ is attained at an interior point $\zeta_\varepsilon>0$, i.e.,
\begin{equation}\label{eq:critzeta}
   Q_\varepsilon=\zeta_\varepsilon^{\cstar-2}A_\varepsilon
   +\mu\,\zeta_\varepsilon^{p-1}B_\varepsilon .
\end{equation}

Moreover, $\zeta_\varepsilon$ is bounded away from $0$. Suppose not. Then there
is a sequence $\varepsilon_k\to0$ with $\zeta_{\varepsilon_k}\to0$. Since $A_{\varepsilon_k}$
and $B_{\varepsilon_k}$ are bounded, the right-hand side of \eqref{eq:critzeta} tends to
$0$, while the left-hand side $Q_{\varepsilon_k}\to S^{N/2}>0$, a contradiction. Hence there
are $\zeta_0>0$ and $\varepsilon_1>0$ with $\zeta_\varepsilon\ge\zeta_0$ for
$0<\varepsilon<\varepsilon_1$.

Furthermore,
\begin{equation}\label{eq:supsplit}
   \sup_{\zeta\ge0}J_\mu(\zeta u_\varepsilon)
   =h_\varepsilon(\zeta_\varepsilon)
    -\frac{\mu\,\zeta_\varepsilon^{p+1}}{p+1}B_\varepsilon
   \le\max_{\zeta\ge0}h_\varepsilon(\zeta)
    -\frac{\mu\,\zeta_0^{p+1}}{p+1}B_\varepsilon .
\end{equation}
Also, for some $\zeta_* > 0$, we have
\begin{equation*}
   \max_{\zeta\ge0}h_\varepsilon
   =\zeta_*^2\Bigl(\frac{Q_\varepsilon}{2}-\frac{\zeta_*^{\cstar-2}A_\varepsilon}{\cstar}\Bigr)
   =\Bigl(\frac12-\frac1{\cstar}\Bigr)\zeta_*^2Q_\varepsilon
   =\frac1N\Bigl(\frac{Q_\varepsilon}{A_\varepsilon}\Bigr)^{\frac{N-2}{2}}Q_\varepsilon
   =\frac1N\bigl(Q_\varepsilon A_\varepsilon^{-2/\cstar}\bigr)^{\frac N2}.
\end{equation*}
By \eqref{eq:QAexp}, $Q_\varepsilon A_\varepsilon^{-2/\cstar}=S+\delta_\varepsilon$ with
$|\delta_\varepsilon|\le C\bigl(|R_\varepsilon|+\varepsilon^N\bigr)$, and by the mean value
theorem $(S+\delta_\varepsilon)^{N/2}\le S^{N/2}+C|\delta_\varepsilon|$ for $\varepsilon$
small. Combining these with \eqref{eq:supsplit}, \eqref{eq:Bbelow} and
$\varepsilon^N\le\varepsilon^{\min\{2,N-2\}}$, we obtain for $\varepsilon$ small,
\begin{equation}\label{eq:supJ}
   \sup_{\zeta\ge0}J_\mu(\zeta u_\varepsilon)
   \le\frac1NS^{\frac N2}
   +C\bigl(\varepsilon^{\sigma}+\varepsilon^{\min\{2,N-2\}}\mathcal M_N(\varepsilon)\bigr)
   -c\,\varepsilon^{\beta},
   \qquad c:=\frac{\mu\,\zeta_0^{p+1}c_0}{p+1}>0 .
\end{equation}
Recall that $\mathcal M_N(\varepsilon)=1$ for
$N\ne4$ and $\mathcal M_4(\varepsilon)=|\log\varepsilon|$. The right-hand side of
\eqref{eq:supJ} is strictly below $\frac1NS^{N/2}$ for small $\varepsilon$ as
soon as
\begin{equation*}
   \varepsilon^{\sigma}=o(\varepsilon^\beta)
   \quad\text{and}\quad
   \varepsilon^{\min\{2,N-2\}}\mathcal M_N(\varepsilon)=o(\varepsilon^\beta)
   \quad \text{ as } \varepsilon\to0.
\end{equation*}
Since $\sigma\le\min\{2,N-2\}$, with $\sigma<2$ when $N=4$, both of the above follow from
$\beta<\sigma$.

If $N+2s>4$, then $\sigma=2-2s$ and
\begin{equation*}
   \beta<2-2s\iff p>1+\frac{4s}{N-2}.
\end{equation*}
If $N=3$ and $s\le\frac12$, then $\sigma=1-s$ and
\begin{equation*}
   \beta<1-s\iff3-\frac{p+1}{2}<1-s\iff p>3+2s .
\end{equation*}
The two cases together give \eqref{eq:prange}. This range is nonempty, because
$1+\frac{4s}{N-2}<\frac{N+2}{N-2}$ for $s<1$, and $3+2s<5$ for $s<1$.
\end{proof}

\begin{proof}[Proof of Theorem~\ref{thm:perturbMu}]
(i) Let $\mu>0$. By Lemma~\ref{lem:verifyCompactness}, $u_0=u_\varepsilon$ satisfies
\eqref{CompactnessAssumptionOnFunctional} for $\varepsilon$ small, and
Theorem~\ref{CriticalExponentMainThm} gives a nontrivial nonnegative radial solution.

\smallskip
(ii) Fix a nonnegative radial $u_0\in C_c^\infty(\bn)\setminus\{0\}$ and put
$Q=\|u_0\|_\lambda^2+[u_0]_s^2>0$ and $B=\int_{\bn}|u_0|^{p+1}\dvg>0$. Dropping the critical
term, which is nonpositive, for $\zeta\ge0$,
\begin{equation*}
   J_\mu(\zeta u_0)\le h(\zeta):=\frac{\zeta^2}{2}Q-\frac{\mu\,\zeta^{p+1}}{p+1}B .
\end{equation*}
Since $h'(\zeta)=\zeta\bigl(Q-\mu\zeta^{p-1}B\bigr)$, the function $h$ increases on
$[0,\zeta_*]$ and decreases on $[\zeta_*,\infty)$, where $\zeta_*^{p-1}=\frac{Q}{\mu B}$.
Hence
\begin{equation*}
   \sup_{\zeta\ge0}J_\mu(\zeta u_0)\le h(\zeta_*)
   =\Bigl(\frac12-\frac1{p+1}\Bigr)\zeta_*^2\,Q
   =\frac{p-1}{2(p+1)}\,Q^{\frac{p+1}{p-1}}\,(\mu B)^{-\frac{2}{p-1}} .
\end{equation*}
The right hand side is strictly decreasing in $\mu$ and equals $\frac1NS^{N/2}$ exactly at
\begin{equation*}
   \mu_0:=\frac1B\Bigl(\frac{N(p-1)}{2(p+1)}\Bigr)^{\frac{p-1}{2}}
   Q^{\frac{p+1}{2}}\,S^{-\frac{N(p-1)}{4}} .
\end{equation*}
So for every $\mu>\mu_0$,
$\sup_{\zeta\ge0}J_\mu(\zeta u_0)<\frac1NS^{N/2}=\frac1NS_{\lambda,s}^{N/2}$, by
Theorem~\ref{thm:sobolevConstants}\textup{(i)}. Thus $u_0$ satisfies
\eqref{CompactnessAssumptionOnFunctional}, and Theorem~\ref{CriticalExponentMainThm}
gives a nontrivial nonnegative radial solution.

\smallskip
(iii) Let $\mu\le0$ and let $u$ be a nontrivial weak solution of \eqref{mainEq2}.
By \eqref{eq:Jprimeuu},
\begin{equation*}
   X:=\|u\|_\lambda^2+[u]_s^2=\tilde A+\mu \tilde B,\qquad
   \tilde A:=\int_{\bn}|u|^{\cstar}\dvg,\quad \tilde  B:=\int_{\bn}|u|^{p+1}\dvg .
\end{equation*}
Since $u\ne0$, we have $X>0$, and $\mu \tilde B\le0$ gives $\tilde A\ge X>0$. By
Theorem~\ref{thm:sobolevConstants}\textup{(i)}, $S_{\lambda,s}=S$ is not attained, so
$X>S\,\tilde A^{2/\cstar}$. Hence $\tilde A>S\,\tilde A^{2/\cstar}$, that is,
$\tilde A^{2/N}>S$. Therefore $u\notin\mathcal B$. For $\mu=0$ this is
Theorem~\ref{thm:nonexistCrit}\textup{(i)}.
\end{proof}
 \section{Appendix}\label{appendixSection}
\begin{lemma}[\textit{Weak Maximum Principle}]\label{weakMaxLemma}
Let $\lambda<\frac{(N-1)^2}{4}$ and
$L u := -\Delta_{\mathbb{B}^N} u + (-\Delta_{\mathbb{B}^N})^s u - \lambda u$. Let
$u \in H^1(\mathbb{B}^N)$ satisfy $Lu \geq 0$ weakly, that is,
\begin{equation*}
\langle u,\varphi\rangle_\lambda
+\frac12\int_{\bn}\int_{\bn}(u(x)-u(y))(\varphi(x)-\varphi(y))\Ks(d(x,y))\dvg(x)\dvg(y)\ge0
\end{equation*}
for every $\varphi\in H^1(\bn)$ with $\varphi\ge0$. Then $u \geq 0$ in $\mathbb{B}^N$.
\end{lemma}
\begin{proof}
We test the inequality $Lu \geq 0$ with $u^-$, the negative part of $u$. This yields:
\begin{small}
\begin{align*}
0 &\leq -\int_{\mathbb{B}^N} |\nabla_{\mathbb{B}^N} u^-|^2 \, dV_{\mathbb{B}^N}\\ 
&+\frac{1}{2}\int_{\bn}\int_{\bn}  (u(x) - u(y))(u^-(x) - u^-(y)) \mathcal{K}_s(d(x, y)) \, dV_{\mathbb{B}^N}(x) dV_{\mathbb{B}^N}(y) 
+ \lambda \int_{\mathbb{B}^N} |u^-|^2 \, dV_{\mathbb{B}^N}.
\end{align*}
\end{small}
Splitting $u = u^+ - u^-$ and rearranging, we obtain
\begin{small}
\begin{align*}
0 &\leq -\|u^-\|_\lambda^2 - [u^-]^2_s 
- \frac{1}{2}\int_{\bn}\int_{\bn} \left(u^+(x) u^-(y) + u^-(x) u^+(y)\right) \mathcal{K}_s(d(x, y)) \, dV_{\mathbb{B}^N}(x) dV_{\mathbb{B}^N}(y).
\end{align*}
\end{small}
Since each term on the right hand side is non-positive, the inequality implies that all must vanish. In particular, $\|u^-\|_\lambda = 0$ and $[u^-]_s = 0$, which forces $u^- \equiv 0$ in $\mathbb{B}^N$. Hence, $u \geq 0$ in $\mathbb{B}^N$, completing the proof.
\end{proof}

\smallskip

\begin{lemma}\label{nonLinearTermConvergence}
$$
\int_{\bn} (u_j^+)^{2^*-1} v \, \dvg \longrightarrow \int_{\bn} (u_{\infty}^+)^{2^*-1} v \, \dvg \quad \text{for all } v \in H^1(\bn),
$$
where the sequence $ \{u_j\} $ is as in the proof of Theorem~\ref{CriticalExponentMainThm}.
\end{lemma}
\begin{proof}
\begin{equation}
\int_{\bn} (u_j^+)^{2^*-1} v \, \dvg = \int_{B(0,r)} (u_j^+)^{2^*-1} v \, \dvg + \int_{\bn \setminus B(0,r)} (u_j^+)^{2^*-1} v \, \dvg, \label{IntegralBreak}
\end{equation}
where the radius $r> 0$ will be chosen later.
First, we analyze the convergence on the ball $ B(0,r) $ using Vitali’s Convergence Theorem, noting that $ u_j \to u_\infty $ a.e., hence $(u_j^+)^{2^*-1}\to(u_\infty^+)^{2^*-1}$ a.e. For any $ \varepsilon > 0 $, since $ |v|^{2^*} \in L^1(\bn) $, there exists $\delta>0$ such that for every measurable $ \Omega \subset B(0,r) $ with $|\Omega|<\delta$,
$$
\left( \int_{\Omega} |v|^{2^*} \, \dvg(x) \right)^{\frac{1}{2^*}} \leq \frac{\varepsilon}{(M S_{\lambda,2^*-1}^{-1/2})^{2^*-1}},
$$
where $ M $ is a constant such that $ \|u_j\|_\lambda < M $ and $ S_{\lambda,2^*-1} $ is the best constant in the Poincar\'e-Sobolev inequality as given in \eqref{PoinSobIneq}. For such $\Omega$, by Hölder’s inequality, we then have
\begin{align*}
\left|\int_{\Omega} (u_j^+)^{2^*-1} v \dvg\right|\leq  \int_{\Omega}|u_j|^{2^*-1}|v| \dvg\leq  &\left( \int_{\Omega}|u_j|^{2^*}\dvg\right)^{\frac{2^*-1}{2^*}} \left( \int_{\Omega}|v|^{2^*} \dvg\right)^{\frac{1}{2^*}}\\
&\leq  \left(S_{\lambda,2^*-1}^{-1/2}\|u_j\|_{\lambda}\right)^{2^*-1}  \left( \int_{\Omega}|v|^{2^*} \dvg\right)^{\frac{1}{2^*}}\\
&< \varepsilon,
\end{align*}
for every $j$, which shows that $ (u_j^+)^{2^*-1} v $ is uniformly integrable on $ B(0,r) $. Since $B(0,r)$ has finite volume, by Vitali’s Convergence Theorem, we can pass the limit $j \to \infty$ in the first integral of \eqref{IntegralBreak}.\\
\smallskip
\noindent 
Now, we focus on the second term in \eqref{IntegralBreak}. Since $u_j\rightharpoonup u_\infty$ in $H^1(\bn)$, we also have $\|u_\infty\|_\lambda\le M$. By Hölder’s inequality and \eqref{PoinSobIneq}, as above,
\begin{align*}
&\left| \int_{B(0,r)^c}  \left\{ (u_j^+)^{2^*-1} v - (u_{\infty}^+)^{2^*-1} v \right\} \dvg\right|\\
&\leq  \int_{B(0,r)^c} |u_j|^{2^*-1} |v| \dvg+ \int_{B(0,r)^c} |u_{\infty}|^{2^*-1} |v| \dvg\\
&\leq  \left[\left( \int_{\bn} |u_j|^{2^*} \dvg\right)^{\frac{2^*-1}{2^*}}+\left( \int_{\bn} |u_\infty|^{2^*} \dvg\right)^{\frac{2^*-1}{2^*}}\right] \left( \int_{B(0,r)^c}|v|^{2^*} \dvg\right)^{\frac{1}{2^*}}\\
&\leq 2\left(S_{\lambda,2^*-1}^{-1/2}M\right)^{2^*-1}\left( \int_{B(0,r)^c}|v|^{2^*} \dvg\right)^{\frac{1}{2^*}}.
\end{align*}
Since $ |v|^{2^*} \in L^1(\bn) $, for any $ \varepsilon > 0 $, we can choose $ r $ sufficiently large such that
$$
\left| \int_{B(0,r)^c} \left\{ (u_j^+)^{2^*-1} v - (u_{\infty}^+)^{2^*-1} v \right\} \, \dvg \right| < \varepsilon
\quad\text{for every } j .
$$
% Combining this with the convergence on $B(0,r)$, we obtain
% $\limsup_{j\to\infty}\left|\int_{\bn}\left\{(u_j^+)^{2^*-1}-(u_\infty^+)^{2^*-1}\right\}v\,\dvg\right|\le\varepsilon$
% for every $\varepsilon>0$. 
This concludes the proof.\qedhere
\end{proof}

\begin{section}{Statements and Declarations}
\textbf{Acknowledgments.} 
Diksha Gupta is supported by the Prime Minister's Research Fellowship (PMRF ID - 1401219). Authors sincerely acknowledge D. Ganguly for the valuable discussions and insights.

\textbf{Conflict of interest.} Authors certify that there is no actual or potential conflict of interest in relation to this article. 
	
\end{section}

\bibliographystyle{plain}
\bibliography{ref}
\end{document}